\documentclass[11pt]{amsart}
\usepackage{amsmath, amsthm, amssymb}
\usepackage{graphicx}
\usepackage{tkz-euclide}
\usepackage{mathbbol}
\usepackage{txfonts}
\usepackage{hyperref}
\usepackage{color}
\usepackage{thmtools} 
\usepackage{thm-restate} 
\usepackage{accents} 

\newcounter{mnotecount}[section]

\title[Uniform Temple charts and applications]{Existence of uniform Temple charts and applications to null distance}

\author[Meco]{B Meco$^1$}
\author[Sakovich]{A Sakovich$^1$}
\author[Sormani]{C Sormani$^2$$^3$}

\thanks{The second and the third author began this research 
during the Thematic Program on Nonsmooth Riemannian and Lorentzian geometry hosted at the Fields Institute for Research in Mathematical Sciences in 2022. The paper was finalized during the Thematic Program on Geometry and Convergence in Mathematical General Relativity at the Simon's Center for Geometry and Physics in 2025. The research was funded in part by Sormani's NSF grant DMS-1612409 and PSC-CUNY and Sakovich's Swedish Research Council grants dnr. 2016-04511 and 2024-04845}

\newtheorem{thm}{Theorem}[section]
\newcommand{\bt}{\begin{thm}}
\newcommand{\et}{\end{thm}}

\newtheorem{cor}[thm]{Corollary} 

\newcommand{\bc}{\begin{cor}}
\newcommand{\ec}{\end{cor}}

\newtheorem{lem}[thm]{Lemma} 

\newcommand{\bl}{\begin{lem}}
\newcommand{\el}{\end{lem}}

\newtheorem{prop}[thm]{Proposition}
\newcommand{\bp}{\begin{prop}}
\newcommand{\ep}{\end{prop}}

\theoremstyle{definition} 
\newtheorem{defn}[thm]{Definition}

\newcommand{\ben}{\begin{itemize}}
\newcommand{\een}{\end{itemize}}

\newcommand{\bd}{\begin{defn}}    
\newcommand{\ed}{\end{defn}}

\newtheoremstyle{boldremark}{\dimexpr\topsep/2\relax}{}{}{}{\bfseries}{.}{.5em}{}
\theoremstyle{boldremark}
\newtheorem{rmrk}[thm]{Remark}

\theoremstyle{definition}
\newtheorem{example}{Example}[section]

\theoremstyle{definition}
\newtheorem{definition}{Definition}[section]

\newcommand{\thmref}[1]{Theorem~\ref{#1}}

\newcommand{\lemref}[1]{Lemma~\ref{#1}}

\newcommand{\EXP}{\operatorname{EXP}}
\newcommand{\dhat}{\hat{d}}

\newcommand{\be}{\begin{equation}}

\newcommand{\ee}{\end{equation}}
\numberwithin{equation}{section} 

\def\iff{\Longleftrightarrow}
\def\implies{\Longrightarrow}

\begin{document}

\begin{abstract}
    In this paper, we prove that Temple's cylindrical future null coordinate charts 
    can be constructed uniformly and we estimate the gradients of their optical functions. We then apply these charts to study a spacetime $(N,g)$ that has been converted into a definite metric space $(N,\dhat_\tau)$, where $\dhat_\tau$ is the null distance of Sormani and Vega 
    defined using a 
    weak temporal function $\tau$. 
    In particular, 
    we prove that $(N, \dhat_\tau)$ is a rectifiable metric space, where the causal structure is locally encoded by $\tau$ and $\dhat_\tau$. As a consequence, 
    applying a classical theorem of Hawking  and following a technique developed by
    Sakovich and Sormani,
    we can prove a Lorentzian isometry theorem, generalizing our earlier result.
\end{abstract}

\maketitle


\begin{centering}
\hfill\\
\begin{tiny}
$^1$Department of Mathematics, Uppsala University, Box 480, 751 06 Uppsala, Sweden\hfill\\
$^2$CUNY Graduate Center, New York, New York 10016, USA\hfill\\
$^3$Lehman College, Bronx, New York 10468, USA\hfill\\
\end{tiny}
\end{centering}

\newpage

\section{Introduction}
\label{sec:intro} 

Recall that a spacetime is a smooth connected time oriented Lorentzian manifold, $(N^{n+1},g)$, whose metric tensor, $g$, has signature $(-,+,\ldots,+)$. Given a point $q\in N$ and a timelike unit speed geodesic $\eta$ running through $\eta(0) = q$, Temple \cite{Temple-1938} constructed a cylindrical future null coordinate chart, $\Phi_{q,\eta}$, mapping a cylinder, $W_q$, about $\vec 0$ in ${\mathbb R}^{n+1}$, onto a neighborhood $\Phi_{q,\eta}(W_q)$ of $q$ in $N$. In particular, $\Phi_{q,\eta}$ maps the origin to $q$ and it maps the central axis of the cylinder to the timelike geodesic  $\eta$, while a radial line emanating from the axis at height $t$ is mapped to a future null geodesic, $\gamma$, emanating from $\gamma(0)=\eta(t)$.  See the left side of Figure~\ref{fig:Temple-intro}.  We note that although the described Temple charts  are similar in spirit to the standard Fermi-Walker coordinates where radial lines are mapped to spacelike curves (see e.g. Misner, Thorne and Wheeler \cite{MisnerThorneWheeler}), their advantage is that they may be used to recover certain information about the local causal structure, as we will explain below.

\begin{figure}[h] 
    \centering
    \includegraphics[width=.8\textwidth]{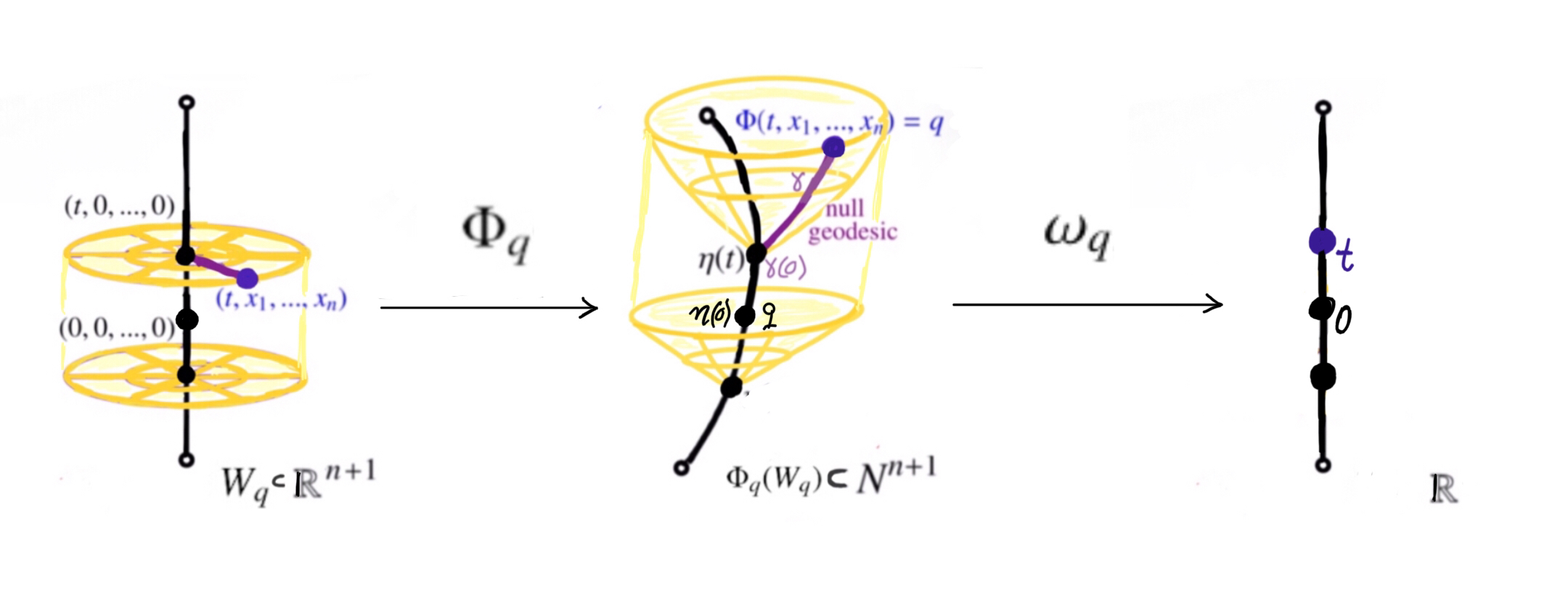} 
    \caption{On the left we see Temple's chart, $\Phi_{q}=\Phi_{q,\eta}:W_q\subset {\mathbb R}^{n+1} \to \Phi_{q}(W_q)\subset N$. On the right we see Temple's optical function $\omega_q=\omega_{q,\eta}$ of this chart, mapping $\Phi_{q}(W_q)$ to an interval and taking the null cones to points.}
    \label{fig:Temple-intro}
\end{figure}

In this article we prove that  Temple charts can be constructed in a uniform way, as stated in the following theorem depicted in Figure~\ref{fig:unif-Temple}:

\begin{thm} 
    \label{thm:unif-Temple-intro}
    Every $p\in N^{n+1}$ has a neighborhood, $U_p\subset N^{n+1}$, with the property that each point $q\in U_p$, has a Temple chart that covers $U_p$ as follows:
    \be 
        \label{eq:unif-chart}
        \Phi_{q,\eta}:W_q \to \Phi_{q,\eta}(W_q) \textrm{ such that } U_p \subseteq \Phi_{q,\eta}(W_q).
    \ee
    Here each $\eta=\eta_q$ with $\eta(0)=q$ is a member of a smooth collection of timelike geodesics perpendicular to a fixed spacelike hypersurface, $\Sigma_p$, that contains $p$.
\end{thm}

\begin{figure}[h] 
    \centering
    \includegraphics[width=.8\textwidth]{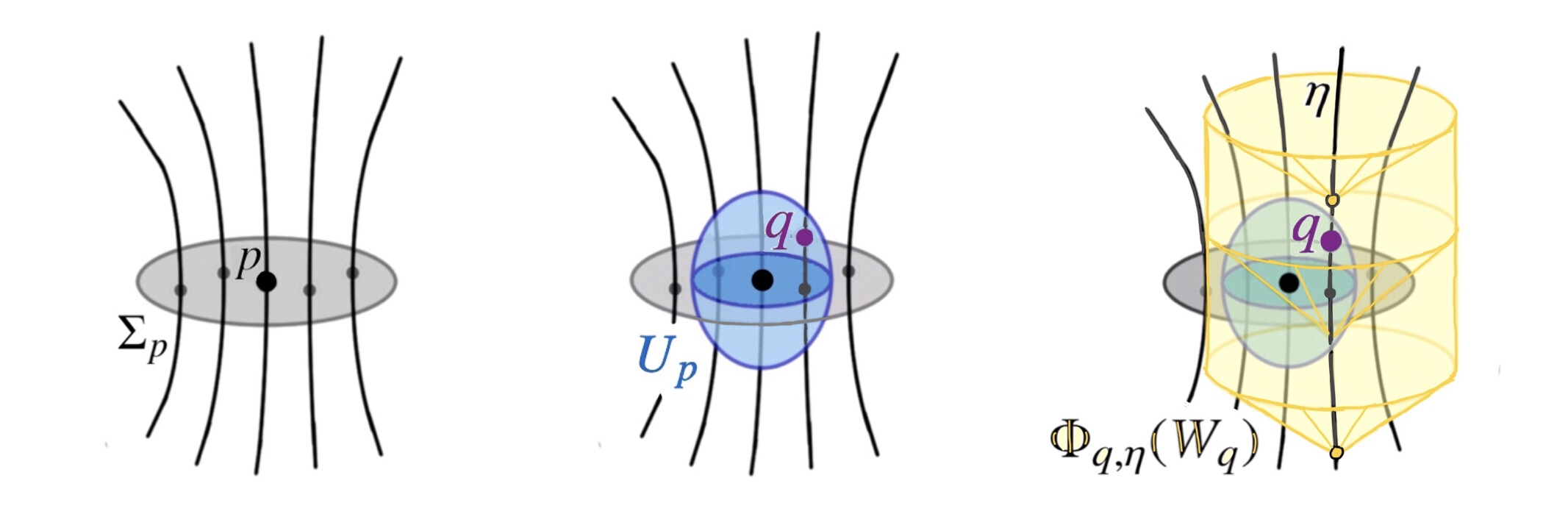} 
    \caption{Theorem \ref{thm:unif-Temple-intro} states that $\forall p\in N$, there exists a spacelike surface, $\Sigma_p$, and a family of  timelike geodesics perpendicular to $\Sigma_p$, and there is a uniform Temple neighborhood, $U_p$, such that $\forall q\in U_p$ there is a timelike geodesic, $\eta = \eta_q$, passing through $q$ and beloning to the aformentioned family such that it has a uniform Temple chart whose image, $\Phi_{q,\eta}(W_q)$, contains $U_p$.}
    \label{fig:unif-Temple}
\end{figure}

Throughout this paper we will refer to any neighborhood, $U_p$, with the properties described in Theorem~\ref{thm:unif-Temple-intro} as a {\em uniform Temple neighborhood} and any of its Temple charts satisfying (\ref{eq:unif-chart}) as a {\em uniform Temple chart}. Our notion of a uniform Temple neighborhood is similar to the notion of a uniform normal neighborhood in Riemannian Geometry, see e.g. Kobayashi and Nomizu \cite[Chapter 3, Theorem 8.7]{KobayashiNomizu}. However, since Temple charts are not diffeomorphisms, we cannot simply apply the Inverse Function Theorem to produce our uniform Temple neighborhood as is done in Riemannian Geometry to prove the existence of uniform normal neighborhoods. Instead, we prove Theorem \ref{thm:unif-Temple-intro} by first showing that for a point \( p \in N \), there is a neighborhood \( V \) of \( p \) and a fixed radius \( r > 0 \) such that for each point \( q \in V \) there is a Temple chart \( \Phi_q \) defined on the cylinder of this radius \( r \). To achieve this, in Section \ref{sec:Prelim} we establish local estimates for Lorentzian geodesics and their Jacobi fields, as well as framed exponential maps and their normal radii, that are subsequently used to define the radius \( r \) as above. After the existence of the neighborhood \( V \) has been proven, the conclusion of Theorem \ref{thm:unif-Temple-intro} follows at once from Brouwer's theorem of invariance of domain. We note that Temple's original construction is described in more detail in Section~\ref{sec:Temple} (see Theorem~\ref{thm-opt}), and our Theorem~\ref{thm:unif-Temple-intro} is restated more precisely as Theorem~\ref{th-uniform-Temple} which is proven in Section \ref{sec:construction}. 

In Section \ref{sec:estimate}, we establish some further properties of uniform Temple charts that will be useful for our applications. The main subject of our study here is Temple's optical function. Although a Temple chart $\Phi_{q,\eta}:W_q\to \Phi_{q,\eta}(W_q)$ is only a smooth diffeomorphism away from the central axis, it is a homeomorphism and has a continuous inverse (see Theorem \ref{th-uniform-Temple}). The first component of $\Phi_{q,\eta}^{-1}$ defines \emph{Temple's optical function} associated with the chart, $\omega_{q,\eta}: \Phi_{q,\eta}(W_q)\to {\mathbb R}$, such that
\be 
    \label{eq:opt}
    \omega_{q,\eta}(\gamma(s))=t \textrm{ when } \gamma=\gamma(s) \textrm{ is a null geodesic with } \gamma(0)=\eta(t).
\ee 
The function $\omega_{q,\eta}$ is continuous everywhere and smooth away from $\eta$. A key property of Temple's optical function is that it can be used to recover the causal 
future of points lying on the central curve:
\be 
    \label{eq:opt-causal}
    J^+(\eta(t) , \Phi_{q,\eta}(W_q))
    = \omega_{q,\eta}^{-1}([t,\infty))\, \cap \, \Phi_{q,\eta}(W_q),
\ee 
as depicted in Figure~\ref{fig:Temple-intro} on the right.  In order to prove useful estimates for the optical function $ \omega_{q,\eta}$, we define a \emph{Riemannianized metric tensor}\footnote{This term was coined by Vega in \cite{Vega21}.}, 
\be
    g_R = g + 2g(e_0,\cdot)g(e_0,\cdot),
\ee 
where $e_0$ is the smooth vector field defined by the tangent vectors to the timelike geodesics used in our construction of the uniform Temple neighborhood. We prove Proposition~\ref{prop-g_R} which provides a uniform estimate for the $g_R$-gradients of the optical functions, $|\nabla^{g_R}\omega_{q,\eta}|_{g_R}$, away from the curve $\eta=\eta_q$. This, in turn, is used to show that each $\omega_{q,\eta}$ is Lipschitz on a uniform Temple neighborhood, see Corollary ~\ref{omega-Lip}.  

In Section \ref{sec:application} we use the uniform Temple charts that we constructed in Section \ref{sec:Temple} to prove new results that we briefly summarize below. In Section~\ref{sec:review}, we prepare the reader for our applications. Here we review the notion of a \emph{time function}, $\tau:N \to {\mathbb R}$, which is a continuous function that is strictly increasing along future causal curves, and of a \emph{weak temporal} function in the sense of Burtscher and Garc\'ia-Heveling \cite{Burtscher-Garcia-Heveling-22}. The later notion includes, for example, the \emph{regular cosmological time function} studied by Andersson, Galloway and Howard \cite{AGH} (see also Wald and Yip \cite{Wald-Yip}). We also review the definition of the {\em null distance}, 
\be
    \hat{d}_\tau: N \times N \to [0,\infty)
\ee 
which, according to Sormani and Vega \cite{SV-null}, can be used to convert a spacetime, $(N,g)$, endowed with a weak temporal function, $\tau: N \to {\mathbb R}$, into a definite metric space, $(N,\hat{d}_\tau)$, with the same topology as the original manifold, $N$.   

In Section \ref{sec:biLip} we build upon estimates of Section \ref{sec:estimate} and show that our uniform Temple charts are bi-Lipschitz:

\begin{thm}
    \label{thm:Lip-intro}
    Let $(N,g)$ be a spacetime equipped with a weak temporal function \( \tau \) 
    and let $\dhat_\tau$ be the associated null distance. Given $p \in N$, there is a uniform Temple neighborhood $U_p$ of $p$ such that for any of the uniform Temple charts  $\Phi_q: W_q \to \Phi_q(W_q)$ centered at $q\in U_p$ and covering $U_p$, the restriction $\Phi_q: (\Phi_q^{-1}(U_p), d_{\mathbb{E}^{n+1}}) \to (U_p,\dhat_\tau)$, where $d_{\mathbb{E}^{n+1}}$ denotes the Euclidean distance on $\Phi_q^{-1}(U_p)\subset \mathbb{R}^{n+1}$, is bi-Lipschitz.
\end{thm}

This result implies, in particular, that $(N,\dhat_\tau)$ is a countably rectifiable metric space, see Corollary~\ref{zero-measure}, which is in fact more general.  

In Section~\ref{sec:null} we apply our uniform Temple charts to recover the local causal structure of a spacetime that has been converted into a metric space using the null distance. Recall that Sakovich and Sormani proved in \cite[Theorem 4.1]{Sak-Sor-null} that if a spacetime, $(N,g)$, has a \emph{proper} locally anti-Lipschitz time function, $\tau$,  then it can be converted canonically into a metric space, $(N, \hat{d}_\tau)$, that \emph{encodes causality globally} in the sense that for all $ q,q'\in N$
we have
\be 
    \label{eq:encodes-causality-introd}
    \hat{d}_\tau(q',q) = \tau(q') - \tau(q) \iff q' \in J^+(q).
\ee
This result was subsequently extended by Burtscher and Garc\'ia-Heveling in \cite{Burtscher-Garcia-Heveling-22} and by Galloway in \cite{Galloway-Null} using other global hypotheses on $(N,g)$ and $\tau$. A key step in the proof is the result of  Sakovich and Sormani \cite[Theorem 1.1]{Sak-Sor-null} which shows that $\hat{d}_\tau$ encodes causality locally in the sense that each $q\in N$ has a neighborhood $U_q$ such that (\ref{eq:encodes-causality-introd}) holds for all $q'\in U_q$. When $\tau$ is a weak temporal function we can establish a stronger version of local encoding of causality stating that each point $p\in N$ has a uniform Temple neighborhood $U_p$ such that (\ref{eq:encodes-causality-introd}) holds for all $q,q'\in U_p$, see Theorem \ref{th-encodes-causality-locally}.

In Section~\ref{sec:Isometry} we apply Theorem \ref{th-encodes-causality-locally} to prove that the conversion of a spacetime $(N,g)$ into a metric space $(N,\dhat_\tau)$ is a unique reversible process up to a choice of a time function $\tau$ satisfying $|\nabla^g \tau|_g =1$. Previously, Sakovich and Sormani applied their result on local encoding of causality  within a Temple chart \cite[Theorem 1.1]{Sak-Sor-null} and a  well-known theorem of Hawking to prove that if a pair of spacetimes, $(N^{n+1}_i,g_i)$, $i=1,2$, with $n\ge 2$ and proper regular cosmological time functions, $\tau_i:N^{n+1}_i\to{\mathbb R}$, have a bijection, $F: N_1\to N_2$, that preserves time, 
\be 
    \label{eq:pres-time}
    (\tau_2 \circ F)(p) = \tau_1(p) \quad \text{ for all } p \in N_1,
\ee
and preserves distance, 
\be 
    \label{eq:pres-dist}
    \hat{d}_{\tau_2}(F(p),F(p')) = \hat{d}_{\tau_1}(p,p') 
    \quad \text{ for all } p,p' \in N_1,
\ee
then $F$ is a diffeomorphism and Lorentzian isometry, see  \cite[Theorem 1.3]{Sak-Sor-null}. In fact, this result holds more generally  assuming that $\tau_i$ are Lipschitz time functions with $|\nabla^{g_i} \tau_i|_{g_i} = 1$ almost everywhere such that $\dhat_{\tau_i}$ encode causality globally. 

Here we remove the hypotheses that causality is globally encoded and prove:  

\begin{thm} \label{thm:isometry-intro}
    Suppose that a pair of spacetimes, $(N^{n+1}_i,g_i)$, $i=1,2$, with $n\geq 2$, have weak temporal functions, $\tau_i:N^{n+1}_i\to{\mathbb R}$, such that $|\nabla^{g_i} \tau_i|_{g_i} =1$ almost everywhere. If a bijection, $F: N_1\to N_2$, preserves time as in (\ref{eq:pres-time}) and preserves distances as in (\ref{eq:pres-dist}), then it is a diffeomorphism and Lorentzian isometry. 
\end{thm}

Our Theorem~\ref{thm:isometry-intro} is a direct consequence of the existence of  uniform Temple charts established in Theorem~\ref{thm:unif-Temple-intro} combined with the local encoding of causality for weak temporal functions in Theorem \ref{th-encodes-causality-locally}, and our following local to global theorem:
\begin{thm} \label{thm:local-to-global}
    Suppose that a pair of spacetimes, $(N^{n+1}_i,g_i)$, $i=1,2$, with $n\geq 2$, have weak temporal functions, $\tau_i:N^{n+1}_i\to{\mathbb R}$, such that $|\nabla^{g_i} \tau_i|_{g_i} =1$ almost everywhere. Suppose additionally that $\tau_i$ and $\dhat_{\tau_i}$ locally encode causality in the sense that
    \be
        \forall p \in N_i \,\,\exists U_p \textrm { such that (\ref{eq:encodes-causality-introd}) holds with } \tau=\tau_i \text{ and } \dhat_\tau = \dhat_{\tau_i} \,  \forall q,q'\in U_p.
    \ee
    If a bijection, $F: N_1\to N_2$, preserves time as in (\ref{eq:pres-time}) and preserves distances as in (\ref{eq:pres-dist}), then it is a diffeomorphism and Lorentzian isometry.  
\end{thm}

The results obtained in this article will be applied in our upcoming work \cite{future-work} on \emph{spacetime intrinsic flat distance} following the strategy outlined in \cite{Sormani-Oberwolfach-18} and \cite{Sak-Sor-24}.

\subsection*{Acknowledgements} We thank the two anonymous referees for their constructive feedback which has helped to improve the manuscript substantially. We also thank Annegret Burtscher and Thijs de Kok for pointing out an inaccuracy in Lemma 4.10 in an earlier version of the paper.

\section{Preliminaries} \label{sec:Prelim}

\subsection{Notation and conventions}\label{sec:notations}

Recall that a \emph{spacetime} is understood as a smooth connected time orientable $(n+1)$-dimensional Lorentzian manifold, $(N^{n+1},g)$, where $n \geq 1$ and $g$ has signature $(-,+,\ldots,+)$. In the sequel, we will use letters $a,b,c,\ldots$ to denote indices within the range $\{0,1, \ldots, n\}$ and we will use letters $i,j,k,\ldots$ to denote indices in the range $\{1,\ldots, n\}$. 

As is standard, we say that a piecewise smooth curve $\gamma: I \to N$ is future causal (respectively null, timelike) if all tangent vectors $\dot{\gamma}$, including the one sided tangents at any breaks or end points, are future causal (respectively null, timelike). Past directed causal (respectively null, timelike) curves are defined analogously. We will use the notation $J^+(p)$ respectively $J^-(p)$ to denote causal future respectively causal past of the point $p$, consisting of all points that can be reached from $p$ by a future respectively past causal curve. We recall that $p\in J^\pm(p)$ by convention. The timelike future respectively timelike past of $p$ is defined similarly and is denoted by $I^+ (p)$ respectively $I^-(p)$. We also recall that the \emph{causal relation} $J^+$ is a subset of $N\times N$ defined by 
\be
J^+:=\{(p, q)\, : \text{ there exists a future-directed causal curve from } p \text{ to } q\}.
\ee
Finally, given an open subset $U\subset N$ for any $p\in U$ we let $J^+(p,U)$ denote the set of all $q\in U$ that can be reached from $p$ by a future causal curve $\gamma: [0,1] \to U$.

The Levi-Civita connection of the metric $g$ will be denoted by $\nabla^g$ or simply by $\nabla$, whenever it causes no confusion. The covariant derivative along a curve $\gamma : I \to N$, $\gamma=\gamma(\lambda)$ will be denoted by $D_\lambda$, and the velocity vector field along the curve will be denoted by $\dot \gamma = \dot \gamma (\lambda)$, for $\lambda \in I$. 

Sometimes we will need to equip an open subset $V\subseteq N$ with a (semi-)Riemannian metric $h$ different from $g$. In this regard, we recall that the gradient $\nabla^h f$ with respect to $h$ of a function $f:V \to \mathbb{R}$ is defined by 
\be
    h(\nabla^h f, X) = X(f) \quad \text{ for any vector field } X.
\ee
Whenever $h$ is a Riemannian metric on $V$, we will use the notation
\be 
    d_{h} (p,q) 
    = d^V_{h} (p,q)
    = \inf_{\begin{subarray}{l}\substack{\gamma:[0,1]\to V}\\ \substack{\gamma(0)=p} \\  \substack{\gamma(1)=q} \end{subarray}}
    \int_0^1 \sqrt{h(\dot \gamma(t),\dot \gamma(t))} \, dt
\ee
to denote the Riemannian distance induced on $V$ by $h$. Whenever $V \subseteq N$ is equipped with a distance function $d$ (for example, $d=d_h$ as above), we will use the notation $B_d(x,r)$ to denote a metric ball of radius $r>0$  centered at $x\in V$. A ball of radius $R>0$ centered at the origin of the Euclidean space $(\mathbb{E}^k, d_{\mathbb{E}^k})$ will be denoted by $B^k(R)$.

The following definition to be used in the sequel is from \cite[Chapter $3$]{Lee}.
\begin{defn} 
    \label{defn:frame-field}
    Given a semi-Riemannian manifold $(N,g)$, a collection of smooth vector fields $\{e_a\} = \{e_0, \ldots ,e_n\}$ defined on an open subset $V \subset N$ is called a \emph{frame field}  if 
    \be
        |g(e_a,e_b) |=  \delta_{ab}, \quad \text{for all $a,b = 0,\ldots, n$}.
    \ee
\end{defn} 

Now suppose that $(N,g)$ is a Lorentzian manifold. Given a frame field $\{e_a\} = \{e_0,\dots,e_n\}$ on an open set $V \subset N$, such that $e_0$ is timelike and $e_1, \dots, e_n$ are spacelike,  one can  \emph{Riemannianize}  $g$ by defining a Riemannian metric  
\be 
    \label{eq:Riemannization}
    g_R(X,Y) := 2g(X,e_0)g(Y,e_0) + g(X,Y).
\ee
We note that the frame field $\{e_a\}$ remains  orthonormal with respect to $g_R$:
\begin{equation}
    \label{eq:g_R-orthog}
    g_R(e_a,e_b) = \delta_{ab}, \quad \text{for all $a,b = 0,\ldots, n$}.
\end{equation}

\subsection{Convex normal neighborhoods and framed exponential maps}
\label{subsec:ConvexNormal}

Given any $(p,v) \in TN$ (which means $v\in T_pN$ is a tangent vector to $N$ at $p$), there exists a unique geodesic, $\gamma:(-\epsilon, \epsilon)\to N$ such that $\gamma(0) = p$ and $ \dot \gamma(0) = v$ and we can define the \emph{exponential map at $p$} by $\exp_p(sv) = \gamma(s)$. By the Fundamental Theorem of Ordinary Differential Equations this defines a smooth function on an open neighborhood 
of the zero section of the tangent bundle $TN$. The Inverse Function Theorem can be applied to prove that for every $p \in N$ there is a neighborhood $\tilde V_p \subset N$ about $p$ and a neighborhood $\tilde V_p^T \subset TN$ about $(p,0) \in TN$ such that the \emph{exponential map}
\be
    \label{exp-p-v}
    \exp: \tilde V_p^T \to \tilde V_p \times \tilde V_p, \text{ defined by } \exp(q,v) = (q,\exp_q v),
\ee
is a diffeomorphism. Note that we have
\be
    \tilde V_p^T = \{(q,v) \, : \, q \in \tilde V_p, \, \exp_q(v) \in \tilde V_p\}.
\ee
Furthermore, one can assume without loss of generality that $\tilde V_p$ is geodesically convex, and that $\tilde V_p^T$ is star shaped, so that 
\be
    \forall q,q' \in \tilde V_p \,\,\exists ! \text{ geodesic } \
    \gamma:[0,1]\to \tilde V_p 
    \textrm{ such that } \gamma(0) = q 
    \textrm{ and } \gamma(1) = q'.
\ee

Note that while $\gamma$ is the unique geodesic within $\tilde V_p$ joining $q$ and $q'$, other geodesics may exist outside the neighborhood $\tilde V_p$, which we call a \emph{convex normal neighborhood} of $p$. For more details see \cite[Chapter $3$, Theorem $8.7$]{KobayashiNomizu}. 
In what follows, we will reserve the notation $\tilde V_p$ and $\tilde V_p^T$ to denote the open sets constructed above. 

The following lemma is an immediate consequence of Definition \ref{defn:frame-field}. 

\begin{lem} 
    \label{lem:framed-map}
    Suppose that $V \subset N$ is an open set and let $\{e_a\}$ be a frame field on $V$. Then there is a natural {\bf frame map}
    \be
        E^{\{e_a\}}: V\times \mathbb R^{n+1} \to TN
    \ee
    defined by 
    \be 
        \label{eq:frame-map}
        E^{\{e_a\}}(q, y_0,\ldots, y_n)=(q,y_0 e_0 +\ldots + y_n e_n)\subset TN
    \ee
    which is a smooth diffeomorphism onto its image.
\end{lem}

\begin{proof}
    The map $E^{\{e_a\}}$ is a diffeomorphism because each vector field $e_a$ for $a = 0,\dots, n$ is smooth and $\{e_0(q),\dots, e_n(q)\}$ is a basis of $T_qN$ at each $q \in V$.
\end{proof}

Recalling the definitions of the sets $\tilde V_p$ and $\tilde V_p^T$ above, we obtain the following lemma.

\begin{lem} 
    \label{lem:framed-exp}
    Let $p \in N$, let $V\subset \tilde V_p$ be a neighborhood of $p$, and let $\{e_a\}$ be a frame field defined on $V$. Then the {\bf framed exponential map at $q \in V$},
    \be
        \EXP^{\{e_a\}}_q: \{(y_0,\ldots,y_n):\, \exp_q(y_0e_0 + \ldots +y_ne_n) \in \tilde V_p\} \to \tilde V_p \subset N^{n+1},
    \ee
    defined by 
    \be
        \EXP^{\{e_a\}}_q(y_0,y_1,\ldots, y_n) := \exp_q(y_0e_0 + y_1e_1 + \ldots +y_ne_n),
    \ee
    is a diffeomorphism. The {\bf framed exponential map} defined by
    \be
        \EXP^{\{e_a\}}(q,y_0,\ldots,y_n) := (q,\EXP^{\{e_a\}}_q(y_0,\ldots,y_n)) 
    \ee
    is smooth in a neighborhood of \( (p,\vec 0) \in N \times \mathbb R^n \) and the restricted maps 
    \begin{align}
        \EXP^{\{e_a\}}&:(E^{\{e_a\}})^{-1}(TV \cap \tilde V_p^T) \to V \times \tilde V_p, \quad \text{and } \label{eq:EXP-1}\\
        \EXP^{\{e_a\}}&: (\EXP^{\{e_a\}})^{-1}(V \times V) \to 
        V \times V,\label{eq:EXP-2}
    \end{align}
    are diffeomorphisms.
\end{lem}

\begin{proof}
    The maps $\EXP^{\{e_a\}}$ in (\ref{eq:EXP-1}) and \eqref{eq:EXP-2} are diffeomorphisms because 
    \be
        \EXP^{\{e_a\}} = \exp \circ E^{\{e_a\}}
    \ee
    and the map $\exp$ as defined in \eqref{exp-p-v} is a diffeomorphism for this choice of domain and range. Similarly, the map $\EXP^{\{e_a\}}_q$ is a diffeomorphism because of the equality $\EXP^{\{e_a\}}_q = \exp_q \circ E^{\{e_a\}}(q,\cdot)$.
\end{proof}

When there is no risk of confusion, we will simply denote the frame map, $E^{\{e_a\}}$, by $E$, the framed exponential map, $\EXP^{\{e_a\}}$, by $\EXP$ and the framed exponential map at $q$, $\EXP^{\{e_a\}}_q$, by $\EXP_q$. 

For proving the existence of a uniform Temple chart as described in the introduction, we will need a suitable Lorentzian analogue of a normal radius in Riemannian geometry. We define one using frame fields. Recall that we denote the open ball of radius $R$ centered at the origin of the Euclidean space $(\mathbb{E}^{n+1}, d_{\mathbb{E}^{n+1}})$ by $B^{n+1}(R)$.  

\begin{lem} 
    \label{lem:normal-rad}
    Suppose that $V \subset \tilde V_p$ is open and that $\{e_a\}$ is a frame field on $V$. For every $q \in V$, there is a largest possible radius $R > 0$ such that the map
    \be
        \EXP_q : B^{n+1}(R) \to \EXP_q(B^{n+1}(R)) \subset V
    \ee
    is a diffeomorphism. This $R > 0$ will be called the {\bf normal radius at $q$} and will be denoted by $R_{\mathrm{N}}(q, V, \{e_a\})$.
\end{lem}

\begin{proof}
    Since $V$ is open and $\EXP_q$ is continuous, $\EXP_q^{-1}(V) \subset \mathbb R^{n+1}$ is open as well. Since $\EXP_q(\vec 0) = q$ it follows that $\vec 0 \in \EXP_q^{-1}(V)$, and by the continuity of \( \EXP_q \) there exists \( \delta > 0 \) such that \( B^{n+1}(\delta) \subset \EXP_q^{-1}(V)\). Letting
    \be
        R_0 = R_{\mathrm{N}}(q,V,\{e_a\}) := \sup\{R > 0: B^{n+1}(R) \subset \EXP_q^{-1}(V)\},
    \ee
    it follows that  $\EXP_q(B^{n+1}(R_0)) \subset V$ for \(R_0 > 0 \). Since $V \subset \tilde V_p$ we may apply Lemma \ref{lem:framed-exp} to see that $\EXP_q$ is a diffeomorphism from $B^{n+1}(R_0)$ to its image.
\end{proof}
Next we show that $R_{\mathrm{N}}(q,V,\{e_a\})$ can be chosen uniformly with respect to $q$ when $q$ is restricted to a compact subset  $K\subset \tilde V_p$. 
\begin{lem}
    \label{lem:EXPDiffeo}
    Suppose that $V \subset \tilde V_p$ is open and that $\{e_a\}$ is a frame field on $V$. Then for every compact subset $K \subset V$ there is a positive radius 
    \be
        R_\mathrm{N}(K, V, \{e_a\}) := \inf_{q \in K} R_\mathrm{N}(q, V, \{e_a\}) > 0,
    \ee
    called the {\bf normal radius of $K$} such that for $R_0 := R_\mathrm{N}(K,V,\{e_a\})$ and all $q \in K$ the maps 
    \be
        \EXP_q : B^{n+1}(R_0) \to \EXP_q(B^{n+1}(R_0)) \subset V
    \ee
    are diffeomorphisms. 
\end{lem}

\begin{proof}
    Suppose on the contrary that $\inf_{q\in K} R_\mathrm{N}(q,V,\{e_a\}) = 0$, then there exists a sequence of points $q_j \in K$ such that $R_\mathrm{N}(q_j,V,\{e_a\}) \to 0$ as $j\to \infty$. Since $K$ is compact, there is a subsequence of $\{q_j\}_j$, denoted for simplicity by the same notation, converging to a point $q\in K$. Due to Lemma \ref{lem:framed-exp} and the inclusion $q\in V \subset \tilde V_p$ we know that $(q,\vec 0)$ lies  in the open set $\EXP^{-1}(V \times V)\subset N \times \mathbb R^{n+1}$. Consequently, there is an open set $U \subset N$ and a radius  $R_U > 0$ such that 
    \be
        (q,0) \in U \times B^{n+1}(R_U) \subset \EXP^{-1}(V \times V).
    \ee
    Now, for all sufficiently large $j$, we have  $q_j \in U$ so
    \be
        (q_j,\EXP_{q_j}(B^{n+1}(R_U)) \subset \EXP(U \times B^{n+1}(R_U)) \subset V \times V,
    \ee
    which implies $\EXP_{q_j}(B^{n+1}(R_U)) \subset V$ and hence the inequality $R_\mathrm{N}(q_j,V,\{e_a\}) \ge R_U > 0$. In particular, we see that $R_\mathrm{N}(q_j,V,\{e_a\})$ cannot converge to $0$, which gives the desired contradiction. 
\end{proof}

We end this section by proving a supplementary result concerning the existence of normal geodesic coordinates adapted to a spacelike hypersurface $\Sigma$ through $p$ and defined on an open subset of its normal convex neighborhood $\tilde V_p$. Although this result is standard, we include a proof since we wish to keep track of various radii involved in later constructions.

\begin{prop} 
    \label{prop:FoliationByGeodesics}
    For every $p\in N$ there is a neighborhood $V$ of $p$, a frame field $\{e_a\}$ defined on $V$ with timelike $e_0$ and spacelike $e_1,\ldots, e_n$, and a radius $R_p > 0$ such that the following holds.
    \begin{enumerate}
        \item The map $F: W_{R_p}: = (-R_p,R_p) \times B^n(R_p) \to V$ defined by
        \be \label{eq:FermiCoordinates}
            F(t,\vec x) = \exp_{\exp_p(x_1e_1 + \cdots + x_ne_n)}(te_0)
        \ee
        is a diffeomorphism.
        
        \item For each $\vec x \in B^n(R_p)$, the curve $t \mapsto F(t,\vec x)$, $t\in (-R_p,R_p) $, is a geodesic for the metric $g$ with the velocity vector field $e_0 $. Every vector field in the frame field $\{e_a\}$ is parallel along this geodesic. 
        
        \item The surface 
        \be
            \Sigma := F(\{0\}\times B^n(R_p))= \EXP^{\{e_a\}}_p(\{0\}\times B^n(R_p))
        \ee
        is a smooth spacelike hypersurface. 
        
        \item The vector field $e_0$ is normal to $\Sigma$ and the vector fields $e_1,\dots, e_n$ are tangent to $\Sigma$ at every point $q \in \Sigma$. 
    \end{enumerate}
\end{prop}

\begin{proof}
    We recall that $\exp: \tilde V_p^T \to \tilde V_p \times \tilde V_p$ is a diffeomorphism and that $\tilde V_p$ is a convex normal neighborhood of $p$. We let $\{e_0,e_1,\ldots,e_n\}$ be any orthonormal basis of $T_pN$ where $e_0$ is a timelike unit vector and $e_1, \ldots, e_n$ are spacelike unit vectors, and define the surface 
    \be
        \Sigma_\sigma := \left\{\exp_p(x_1e_1 + \ldots + x_ne_n): \sum_{i = 1}^n x_i^2 < \sigma^2\right\}.
    \ee
    Assuming that $\sigma > 0$ is small enough, we can ensure that
    \be
        \sum_{i = 1}^n x_i^2 < \sigma^2 \implies x_1e_1 + \dots + x_ne_n \in \exp_p^{-1}(\tilde V_p).
    \ee
    Consequently, there is a $\sigma_0 > 0$, such that  for all $\sigma \in (0,\sigma_0)$ the surface $\Sigma_\sigma$ defined as above is smooth and spacelike. 
    
    Transporting the vectors $\{e_0, e_1, \dots, e_n\}$ parallelly along the radial geodesics $\gamma_u(t) := \exp_p(t(u^1e_1 + \ldots + u^ne_n))$, where $|u|^2=\sum_{i = 1}^n (u^i)^2 = 1$, we obtain a collection of orthogonal smooth vector fields $\{e_0,e_1,\dots, e_n\}$ along $\Sigma_\sigma$ such that $e_0$ is timelike and normal to $\Sigma_\sigma$ everywhere and every $e_i$ is spacelike and tangent to $\Sigma_\sigma$ everywhere.
    
    By the Fundamental Theorem of Ordinary Differential Equations we conclude that the map $F: (-\delta,\delta) \times B^n(\sigma) \to \tilde V_p$ given by
    \be
        F(t,\vec x) := \exp_{\exp_p(x_1 e_1 + \cdots + x_n e_n)}(te_0)
    \ee
    is well defined and smooth for some $\delta > 0$. Observing that
    \be
        DF_{(0,\vec 0)}(\partial_t) = e_0 \quad \text{and} \quad DF_{(0,\vec 0)}(\partial_{x_i}) = e_i, \text{ for $i = 1, \dots, n$,}
    \ee
    it follows that $DF_{(0,\vec 0)}$ is invertible and by the Inverse Function Theorem there is some $R_p \in (0,\min(\delta,\sigma))$ such that $F: W_{R_p}:=(-R_p,R_p) \times B^n(R_p) \to V := F(W_{R_p})$ is a diffeomorphism. This proves the first statement of the proposition, and the third and the fourth statements follows directly from our construction once we define $\Sigma := \Sigma_{R_p}$. To prove the second statement, we only need to extend the vector fields $\{e_0,e_1,\dots, e_n\}$ to $V$ by transporting them parallelly along the geodesics $t \mapsto F(t,\vec x)$, for $\vec x \in B^n(R_p)$. The smoothness of the resulting frame fields follows from standard arguments using theory for linear systems of ODEs, cf. \cite[Chapter 5]{Lee}.
\end{proof}

\subsection{Local estimates for Lorentzian geodesics and their Jacobi fields}
\label{subsec:LocalEstimates}
In this section we prove two technical results that will be used in our construction of uniform Temple charts. The first result is a simple consequence of the smoothness of the framed exponential map.

\begin{lem} 
    \label{lem:GoodGeodesics}
    Let $V\subset \tilde V_p$ be a neighborhood of $p\in N$, let $\{e_a\}$ be a frame field on $V$, and let $K \subset V$ be a compact set. Then, for every $\epsilon > 0$ there is a 
    \be
        \delta := \delta(K,V,\{e_a\},\epsilon) \in (0,\epsilon)
    \ee 
    such that every geodesic $\gamma$ satisfying $\gamma(0) \in K$ and 
    \be 
        \label{eq:GeodesicSmallGradient}
        \max_{0 \leq a \leq n} |g(\dot \gamma(0),e_a)| \leq \delta, 
    \ee
    is defined on the interval $[0,1]$ and satisfies $\gamma([0,1]) \subset V$. Moreover, we have 
    \be
        \max_{0 \leq a \leq n} \sup_{\lambda \in [0,1]}|g(\dot \gamma(\lambda),e_a)| \leq \epsilon.
    \ee
\end{lem}

\begin{proof}
    In this proof, \( \vec x \) denotes a vector in \( \mathbb R^{n+1} \). By Lemma \ref{lem:framed-exp} we know that the map 
    \be
        \EXP: E^{-1}(\tilde V_p^T \cap TV) \to V \times \tilde V_p, \quad (q,x_0,\ldots,x_n) \mapsto (q,\EXP_q( x_0,\ldots,  x_n))
    \ee
    is a diffeomorphism. Consequently, recalling that  $\tilde V_p$ is geodesically convex, it follows that the map 
    \be
        \mathcal D: [0,1]  \times E^{-1}(\tilde V_p^T\cap TV) \to TN, 
    \ee
    given by 
    \be
        \mathcal D(\lambda, q, \vec x) = \left(\EXP_q(\lambda \vec x),\frac{d}{ds}\bigg|_{s = \lambda}\EXP_q(s \vec x)\right),
    \ee
    is well-defined and continuous. Since 
    \be
        \mathcal D(\lambda, q, \vec 0) = (q, \underbrace{0\cdot e_0+\ldots + 0 \cdot e_n}_{=\vec 0_q \in T_q N})
    \ee
    holds for all $\lambda \in [0,1]$ we have  
    \be
        \label{eqZeroSection}
        \mathcal D([0,1]\times K \times \{\vec 0\} ) = \{(p,\vec 0_p) : p \in K, \, \vec 0_p \in T_p N\},
    \ee
    the zero section of the tangent bundle $TK$.
    
    Now let $V$ be an open set containing $K$ as in the statement of the lemma. Given any $\epsilon>0$ we define the \emph{cube of radius $\epsilon$} in $\mathbb{R}^{n+1}$ by
    \be
        \label{eqCube}
        C_\epsilon  = \left\{(x_0,x_1,\dots, x_n)\in \mathbb{R}^{n+1} : \, \max_{0 \leq a \leq n} |x_a| < \epsilon\right\}.
    \ee
    Clearly, an open set 
    \be
        V\times E(C_\epsilon)=\left\{(p,v) \in TV : \, |g(e_a,v)| < \epsilon \text{ for all $a = 0,\dots, n$}\right\}
    \ee
    is a neighborghood of the compact set $\mathcal D([0,1]\times K \times \{\vec 0\} )$, see \eqref{eqZeroSection}, and we have
    \be
        [0,1]\times K \times \{\vec 0\} \subset   [0,1]\times E^{-1}(\tilde V_p^T\cap TV).
    \ee
    Consequently, there is an open neighborhood $W$ of $[0,1]\times K \times \{\vec 0\} $ in $[0,1]  \times E^{-1}(\tilde V_p^T\cap TV)$ such that $\mathcal D(W) \subset V \times E ( C_\epsilon)$. Without loss of generality, we may assume that this neighborhood is of the form $W = [0,1] \times U \times C_{\delta}$ where $U$ is an open set such that $K\subset U \subset V$, and $C_\delta$ is defined by \eqref{eqCube} for some $\delta \in (0,\epsilon)$.  
    
    The above implies that any geodesic $\gamma$ with $\gamma(0) \in K$ such that \eqref{eq:GeodesicSmallGradient} holds is defined on the interval $[0,1]$, satisfies $\gamma([0,1]) \subset V$ and we have 
    \be
        \max_{0 \leq a \leq n}|g(\dot \gamma(\lambda),e_a)| < \epsilon \text{ for all $\lambda \in [0,1]$}
    \ee
    as claimed.
\end{proof}

The following result states, roughly speaking, that an initially timelike Jacobi field along a geodesic within $\tilde V_p$ will stay timelike as long as we remain in the neighborhood.

\begin{restatable}{lem}{JacobiFields}
    \label{lem:JacobiTime}
    Suppose that $K \subset \tilde V_p$ is a compact set, that $\{e_a\}=\{e_0,\ldots,e_n\}$ is a frame field defined on an open neighborhood $V$ of $K$ with timelike $e_0$ and spacelike $e_1,\ldots, e_n$, and that $\gamma: [0,1] \to K$ is a geodesic such that 
    \be 
        \label{eq:GradientCondition}
        \sup_{\lambda \in [0,1]} \max_{0 \leq a \leq n} |g(\dot \gamma(\lambda),e_a)| \leq \epsilon \quad \text{for all $\lambda \in [0,1]$}.
    \ee
    Then there is a constant $C > 0$ depending only on the metric $g$ and on the frame field $\{e_a\}$ restricted to the compact set $K$, such that the Jacobi field $J = J(\lambda)$ along $\gamma$ defined by the initial conditions $J(0) = e_0$ and $D_{\lambda}J(0) = 0$ satisfies
    \be 
        \label{eq:JacobiEstimate}
        \max_{0 \leq a \leq n}\sup_{\lambda \in [0,1]} |g(J-e_0,e_a)| < C\epsilon, \,\text{ and hence also } \, \sup_{\lambda \in [0,1]}|g(J,J) + 1| < C\epsilon.
    \ee
    In particular, there is an $\epsilon_0 = \epsilon_0(K,\{e_a\}) > 0$ such that if \eqref{eq:GradientCondition} holds with $\epsilon \leq \epsilon_0$, then the vector $J(\lambda)$ is timelike for all $\lambda \in [0,1]$.
\end{restatable}

\begin{proof}
    Let $\{\omega^a\}=\{\omega^0,\ldots,\omega^n\}$ denote the dual frame to $\{e_a\}=\{e_0,\ldots,e_n\}$ defined by $\omega^a(e_b)=\delta_b^a$, and let $g^{-1}$ be the dual metric  defined by $\sum_{c=1}^{n+1}g^{-1}(\omega^a,\omega^c) g(e_c,e_b)=\delta_b^a$. As discussed in Section \ref{sec:notations}, we can use the given frame field $\{e_a\}=\{e_0,\ldots,e_n\}$ to equip $V$ with a Riemannian metric $g_R$ defined by 
    \be
        \label{eqRiemNear}
        g_R(X,Y) := 2g(X,e_0)g(Y,e_0) + g(X,Y).
    \ee
    Since $\{e_a\}$ remains orthonormal with respect to $g_R$ (see \eqref{eq:g_R-orthog}) we have 
    \be
        \label{eq:g-g_R-norm}
        |g|_{g_R} =  |g^{-1}|_{g_R} = \sqrt{n}.
    \ee
    The norms here are the usual tensor norms with respect to the Riemannian metric $g_R$ which can be conveniently computed using the frame fields $\{e_a\}$ and their duals $\{\omega_a\}$, see for example \cite[Chapter 3]{O'neill-text}. We note also that we have 
    \be
        \label{eq:g-vs -g_R}
        |g_R(X,e_a)| = |g(X,e_a)| \text{ for any vector field } X \text{ on } \tilde V_p \text{ and } a=0,1,\ldots, n.
    \ee
    Furthermore, given a compact set $K\subset \tilde V_p$ as in the statement of the lemma, there is a constant $C_0 > 0$ depending only  on $g$, $\{e_a\}$, and $K$, such that the Riemann curvature tensor $R^g$ of the smooth Lorentzian metric $g$ and the Christoffel symbols $\Gamma_{ij}^k$ defined by $ \nabla^g_{e_i}e_j = \Gamma_{ij}^k e_k$ satisfy
    \be 
        \label{eq:AppendixBoundsNeededGeometric}
        \max_K |R^g|_{g_R} \leq C_0 \quad \text{ and } \quad \max_{i,j,k \in\{0,\dots, n\}} \max_K  |\Gamma_{ij}^k| \leq C_0.
    \ee
    Throughout the rest of the proof, $C$ denotes a positive constant that may change from line to line but that is only allowed to depend on the above constant $C_0 > 0$ and the dimension of the spacetime $(N,g)$. 

    For computations involving Jacobi fields it is more convenient to use frame fields  that are parallel along their geodesic, rather than the background frame field $\{e_a\}$ which may not have this property. Therefore we will temporarily switch to the orthonormal frame $\{\hat e_a\}=\{\hat e_0, \hat e_1, \dots, \hat e_n\}$ obtained by the parallel transport  (with respect to $g$) of the vectors $e_a(\gamma(0))$, $a = 0,\dots, n$, along the geodesic $\gamma$. Similar to \eqref{eqRiemNear}, we can Riemannianize the metric $g$ along $\gamma$ as follows:
    \be
        \label{eq:RiemannianizedHat}
        \hat g_R(X,Y) := 2g(X,\hat e_0)g(Y,\hat e_0) + g(X,Y) \text{ for } X,Y \in T_{\gamma(\lambda)}N \text{ and } \lambda\in [0,1].
    \ee
    Clearly, $\{\hat e_a\}$ is orthonormal with respect to $\hat g_R$ and we have \be\label{eq:g-vs-hat-g_R}
        |\hat g_R(X,\hat e_a)| = |g(X,\hat e_a)| \text{ for any vector field } X \text{  along } \gamma \text{ and } a=0,1,\ldots, n.
   \ee
    Moreover, since $D_\lambda \hat e_0 = 0$, where $D_\lambda$ denotes the covariant derivative of the metric $g$ along $\gamma$, it is easy to check that 
    \be 
        \label{eq:CovariantHat}
        D_\lambda \hat g_R(X,Y) = \hat g_R(D_\lambda X,Y) + \hat g_R(X,D_\lambda Y) 
    \ee
    for any vector fields $X$ and $Y$  along $ \gamma$.
    
    In order to prove the lemma, we will first establish the analogue of the estimates \eqref{eq:JacobiEstimate} in the frame $\{\hat e_0, \hat e_1, \dots, \hat e_n\}$. For this, we define $f=f(\lambda)$, $\lambda\in [0,1]$, by 
    \be
        f := \langle J,J\rangle_{\hat g_R} + \langle D_\lambda J,D_\lambda J\rangle_{\hat g_R}. 
    \ee
    Clearly, we have $f\geq 0$ and we claim that there is a constant $C$  as described above such that $ f \leq C$ for all $\lambda \in [0,1]$. To see this, we note that
    \be
        \label{eq:DifferentialInequalityF}
        \begin{aligned}
            f'
            & = 2\langle J,D_\lambda J\rangle_{\hat g_R} + 2\langle D_\lambda J,D_\lambda^2 J\rangle_{\hat g_R} \\ 
            & = 2\langle J,D_\lambda J\rangle_{\hat g_R} - 2\langle D_\lambda J,R^g(J, \dot \gamma)\dot \gamma\rangle_{\hat g_R} \\
            & \leq 2|J|_{\hat g_R}|D_\lambda J|_{\hat g_R} + 2|D_\lambda J|_{\hat g_R} |R^g|_{\hat g_R}|J|_{\hat g_R}|\dot \gamma|_{\hat g_R}^2 \\ 
            & \leq (|J|_{\hat g_R}^2 + |D_\lambda J|_{\hat g_R}^2)(1 + |R^g|_{\hat g_R}|\dot \gamma|_{\hat g_R}^2) \\
            & = f(1 + |R^g|_{\hat g_R}|\dot \gamma|_{\hat g_R}^2).
        \end{aligned}
    \ee
    Here we have used \eqref{eq:CovariantHat} in the first line, the Jacobi equation 
    \be 
        \label{eq:JacobiEquation}
        D^2_\lambda J + R^g(J, \dot \gamma) \dot \gamma = 0
    \ee 
    in the second line, the Cauchy-Schwarz inequality for the Riemannian metric \(\hat g_R\) in the third line and the elementary inequality \( 2ab \leq a^2 + b^2 \) in the fourth line.

    Now let us assume for the moment that the analogues of the estimate \eqref{eq:GradientCondition} and the first bound in \eqref{eq:AppendixBoundsNeededGeometric} hold along $\gamma$ with respect to the metric $\hat g_R$, or, more specifically, that we have
    \be 
        \label{eq:EstimateODECoefficient}
        |R^g|_{\hat g_R}\leq C \quad \text{and} \quad |\dot\gamma|_{\hat g_R}^2 \leq C\epsilon^2 \quad \text{on } \gamma
    \ee
    for $C$ as described above (we will verify this claim later). In this case it follows from \eqref{eq:DifferentialInequalityF} that $f' \leq Cf$. Integrating this inequality we obtain 
    \be
        \label{eqTrivialEstimate}
        0 \leq f = |J|^2_{\hat g_R} + | D_\lambda J|^2_{\hat g_R} \leq C, 
    \ee
    recalling our convention for the constant $C$. 
    
    We will now refine this estimate, proving that the analogue of \eqref{eq:JacobiEstimate} holds with respect to the parallel frame $\{\hat{e}_a\}$ along the geodesic $\gamma$. This will be achieved by estimating the coefficients $f_a= \hat g_R(J,\hat e_a)$ in the expansion $J = \sum_{a = 0}^n f_a\hat e_a$, while still assuming \eqref{eq:EstimateODECoefficient} which remains to be proven.  Since each $\hat e_a$ is parallel along $\gamma$ we find that $D_\lambda^2J = \sum_{a = 0}^n f_a''\hat e_a$ and by the Jacobi equation \eqref{eq:JacobiEquation} we have
    \be
        \sum_{a = 0}^n (f_a'')^2 = |D_\lambda^2 J|_{\hat g_R}^2 = |R^g(J,\dot \gamma)\dot \gamma |_{\hat g_R}^2 \leq |J|_{\hat g_R}^2|R^g|_{\hat g_R}^2|\dot \gamma|_{\hat g_R}^4 \leq C\epsilon^4,
    \ee
    where we have used \eqref{eq:EstimateODECoefficient} and \eqref{eqTrivialEstimate}. Thus for all $a = 0,\dots, n$ we have $|f''_a| \leq C\epsilon^2$. Integrating twice, we find that
    \be
        |f_a(\lambda) - f_a(0) -  f_a'(0)\lambda| \leq C\epsilon^2. 
    \ee
    Further, the initial conditions $J(0) = e_0$, $D_\lambda J(0) = 0$ imply that $f_0(0) = 1$, $f_i(0) = 0$ for $i = 1,\dots, n$ and $f_a'(0) = 0$ for $a = 0, \dots, n$. Summing up, we obtain
    \be 
        \label{eq:HatCoefficientsJacobi}
        |f_0(\lambda) - 1| \leq C\epsilon^2, \quad |f_i(\lambda)| \leq C\epsilon^2, \quad \text{ for all } \lambda \in [0,1],
    \ee
    proving the analogue of the estimates \eqref{eq:JacobiEstimate} in the frame $\{\hat e_a\}$.

    Next, we turn to proving the estimates \eqref{eq:EstimateODECoefficient}, that we have so far assumed to hold along $\gamma$. As a first step, we will show that for all $\lambda \in [0,1]$ and $a = 0,\dots, n$, we have
    \be
        \label{eq:AppendixHatNorm}
        |\hat e_a|_{g_R} = \left(\sum_{b = 0}^n g_R(\hat e_a,e_b)^2\right)^{1/2} = \left(\sum_{b = 0}^n g(\hat e_a,e_b)^2\right)^{1/2} \leq C
    \ee
     where the constant $C>0$ is as described above. For this, we note that for all $a = 0,\dots, n$ we have
    \be 
        \label{eq:AppendixParallellTranslatesNorm}
        D_\lambda |\hat e_a|_{g_R}^2 = D_\lambda \sum_{b = 0}^n g(\hat e_a,e_b)^2
        \leq \sum_{b = 0}^n 2|g(\hat e_a,e_b)|\, |D_\lambda g(\hat e_a,e_b)|.
    \ee
    Since each $\hat e_a$ is parallel along $\gamma$ with respect to the Lorentzian metric $g$, for all $a,b = 0,\dots, n$ we have 
    \be 
        \label{eq:AppendixMixedProductBound}
        \begin{split}
            |D_\lambda g(\hat e_a,e_b)|
            & = |g(\hat e_a,D_{\lambda} e_b)| 
            \leq \sum_{c = 0}^n |g(\hat e_a,e_c)|\,|g(e_c,D_{\lambda} e_b)| \\
            & = \sum_{c,d = 0}^n |g(\hat e_a,e_c)|\,|g(\dot \gamma,e_d)|\,|\Gamma_{db}^c|
            \leq C_0\epsilon\sum_{c = 0}^n |g(\hat e_a,e_c)|
        \end{split}
    \ee
    where in the last step we used  \eqref{eq:GradientCondition} and the second bound in \eqref{eq:AppendixBoundsNeededGeometric}. Combining \eqref{eq:AppendixParallellTranslatesNorm} and \eqref{eq:AppendixMixedProductBound} we obtain
    \be
        \begin{aligned}
        D_\lambda |\hat e_a|_{g_R}^2 
        & \leq C\epsilon \sum_{b,c = 0}^n |g(\hat e_a,e_b)|\, |g(\hat e_a,e_c)| \\
        & = C\epsilon \sum_{b,c = 0}^n |g_R(\hat e_a,e_b)|\, |g_R(\hat e_a,e_c)| \\
        & \leq C\epsilon|\hat e_a|_{g_R}^2, 
        \end{aligned}
    \ee
    where we used \eqref{eq:g-vs -g_R} in the second line and the Cauchy-Schwarz inequality for the Riemannian metric \( g_R \) in the third line. Integrating this differential inequality we obtain \eqref{eq:AppendixHatNorm} after recalling that $|\hat e_a(\gamma(0))|_{g_R} = |e_a(\gamma(0))|_{g_R} = 1$.
    
    With \eqref{eq:AppendixHatNorm} at hand it is now straightforward to prove \eqref{eq:EstimateODECoefficient}. Indeed, using \eqref{eq:g-vs-hat-g_R}, \eqref{eq:GradientCondition}, \eqref{eq:g-g_R-norm},  and \eqref{eq:AppendixHatNorm}, we find that
    \be
        |\dot\gamma|_{\hat g_R}^2 = \sum_{a = 0}^n \hat g_R(\dot\gamma, \hat e_a)^2 =  \sum_{a = 0}^n g(\dot\gamma, \hat e_a)^2 \leq \sum_{a = 0}^n |g|_{g_R}^2 \, |\dot \gamma |_{g_R}^2 \, |\hat e_a|_{g_R}^2 \leq C\epsilon^2,
    \ee
    and similarly, using \eqref{eq:g-vs-hat-g_R}, \eqref{eq:g-g_R-norm}, \eqref{eq:AppendixBoundsNeededGeometric} and \eqref{eq:AppendixHatNorm}, we get
    \be
        \begin{split}
            |R^g|_{\hat g_R}^2
            & = \sum_{a,b,c,d = 0}^n \hat g_R(R^g(\hat e_a,\hat e_b)\hat e_c, \hat e_d)^2
            = \sum_{a,b,c,d = 0}^n g(R^g(\hat e_a,\hat e_b)\hat e_c, \hat e_d)^2 \\
            & \leq \sum_{a,b,c,d = 0}^n |g|_{g_R}^2 |R^g|_{g_R}^2 |\hat e_a|_{g_R}^2|\hat e_b|_{g_R}^2|\hat e_c|_{g_R}^2|\hat e_d|_{g_R}^2 \leq C.
        \end{split}
    \ee
    For completing the proof, we need to switch from the frame $\{\hat e_a\}$ along $\gamma$ to the background frame $\{e_a\}$ defined on all of $K$, so that we can transform the estimate \eqref{eq:HatCoefficientsJacobi} to \eqref{eq:JacobiEstimate}. For this, we will need a refined version of the estimate \eqref{eq:AppendixHatNorm}, namely 
    \be 
        \label{eq:AppendixDifferenceNorm}
        |\hat e_a - e_a|_{g_R} = \left(\sum_{b = 0}^n g_R(\hat e_a - e_a,e_b)^2\right)^{1/2} \leq C\epsilon.
    \ee
    The proof is very similar to our previous arguments. Indeed, the bounds \eqref{eq:AppendixMixedProductBound} and \eqref{eq:AppendixHatNorm} together with a Cauchy-Schwartz inequality with respect to $g_R$ and \eqref{eq:g-g_R-norm} imply that $|D_\lambda g(\hat e_a,e_b)| \leq C\epsilon$, for all $a,b = 0,\dots, n$. Consequently, for $a = 0,\dots, n$, we have
    \be
        \begin{split}
            D_\lambda |\hat e_a - e_a|_{g_R}^2 
            & = D_\lambda \sum_{b = 0}^n g(\hat e_a - e_a,e_b)^2 = 2\sum_{b = 0}^n g(\hat e_a - e_a,e_b)D_\lambda g(\hat e_a,e_b) \\
            & \leq C\epsilon\sum_{b = 0}^n |g(\hat e_a - e_a,e_b)|
            \leq C\epsilon \left(\sum_{b = 0}^n g(\hat e_a - e_a,e_b)^2\right)^{1/2},
        \end{split}
    \ee
    so that $D_\lambda |\hat e_a - e_a|_{g_R}^2 \leq C\epsilon|\hat e_a - e_a|_{g_R}$. This yields $D_\lambda |\hat e_a - e_a|_{g_R} \leq C\epsilon$, which together with the initial condition $\hat e_a(\gamma(0)) = e_a(\gamma(0))$ implies \eqref{eq:AppendixDifferenceNorm}.
    
    Finally, we can prove our Lemma, namely \eqref{eq:JacobiEstimate}. Using \eqref{eq:HatCoefficientsJacobi} and \eqref{eq:AppendixHatNorm} it follows that
    \be
        |J|^2_{g_R} = \sum_{a = 0}^n |g_R(J,e_a)|^2=\sum_{a,b = 0}^n |f_b g_R(\hat e_b,e_a)|^2 = \sum_{a,b = 0}^n |f_b|^2\,|g_R(\hat e_b,e_a)|^2 \leq C.
    \ee
    With this bound at hand, using \eqref{eq:RiemannianizedHat}, \eqref{eq:HatCoefficientsJacobi},  and \eqref{eq:AppendixDifferenceNorm}  we find that
    \be
        \begin{split}
            |g(J - e_0,e_0)|
            & = |g(J,e_0) + 1|\\
            & = |g(J,e_0 - \hat e_0) + g(J,\hat e_0) + 1| \\
            & = |g(J,e_0 - \hat e_0) - \hat g_R(J,\hat e_0) + 1| \\
            & \leq |g(J,e_0 - \hat e_0)| + |-f_0 + 1| \\
            & \leq |g|_{g_R} |J|_{g_R} |e_0 - \hat e_0|_{g_R}  + C\epsilon^2 \leq C\epsilon.
        \end{split}
    \ee
    Similarly, for $i = 1,\dots, n$ we have
    \be
        \begin{split}
            |g(J - e_0,e_i)|
            & = |g(J,e_i)| \\
            &= |g(J,e_i - \hat e_i) + g(J,\hat e_i)|\\
            &= |g(J,e_i - \hat e_i) + \hat g_R(J,\hat e_i)|\\
            &\leq |g(J,e_i - \hat e_i)| + |f_i| \\
            & \leq |g|_{g_R} |J|_{g_R} |e_i - \hat e_i|_{g_R} + C\epsilon^2 \leq C\epsilon,
        \end{split}
    \ee
    completing the proof of \eqref{eq:JacobiEstimate}. 
\end{proof}

\section{Uniform Temple Charts} \label{sec:Temple} 
\subsection{Review of Temple Charts}
In 1938 Temple proved the following theorem  which was
stated in modern terminology using exponential maps as follows by Sakovich and Sormani in \cite{Sak-Sor-null}: 

\begin{thm}{\textnormal{\cite[Section 3]{Temple-1938}}}
    \label{thm-opt}
    Given any $q \in N$, let $\eta:(-\epsilon, \epsilon) \to N$ be a unit speed future timelike geodesic through $\eta(0) = q$. Let $e_0 = \dot \eta(0)$ and let $e_1,...,e_n \in T_q N$ be an orthonormal collection of spacelike vectors such that $e_i+e_0$, $i=1,\ldots,n$, is future null. We extend this frame by parallel transport along $\eta$ noting that since $\eta$ is a geodesic, $\dot \eta(t)=e_{0}$ at $\eta(t)$ for all $t\in(-\epsilon,\epsilon)$. 
    
    Noting that for any $x=(x_1,\ldots,x_n) \in \mathbb R^n$ and $t\in (-\epsilon,\epsilon)$ the vector 
    \be
        \sum_{i = 1}^n x_i e_i + |x| \dot \eta(t) \in T_{\eta(t)}N \, \text{ is null},
    \ee 
    we define a {\bf Temple chart} $\Phi_q=\Phi_{q,\eta}$ by
    \be 
        \label{eqTempleClassic}
        \Phi_q(t, \vec x) = \exp_{\eta(t)}\left(|\vec x|e_0 + \sum_{i=1}^n x^i e_i\right).
    \ee
    The chart $\Phi_q:W_q \to \Phi_q(W_q)$ is continuous and invertible on a neighborhood $W_q$ of $(-\epsilon, \epsilon) \times \{0\}^n$ in $\mathbb{R}^{n+1}$ and is smooth in this neighborhood away from $(-\epsilon,\epsilon) \times \{0\}^n$. In this chart, we define the {\bf optical function} $\omega_q = \omega_{q,\eta}$ by
    \be 
        \label{eq:OpticalFunction}
        \omega_q(\Phi_q(t, \vec x)) = t
    \ee
    and the {\bf radial function} $\lambda_q = \lambda_{q,\eta}$ by
    \be 
        \label{eq:RadialCoordinate}
        \lambda_q( \Phi_q(t, \vec x) ) = \sqrt{x_1^2+\cdots+x_n^2}.
    \ee
\end{thm}
\begin{figure}[h]
    \centering
    \includegraphics[width=1\textwidth]{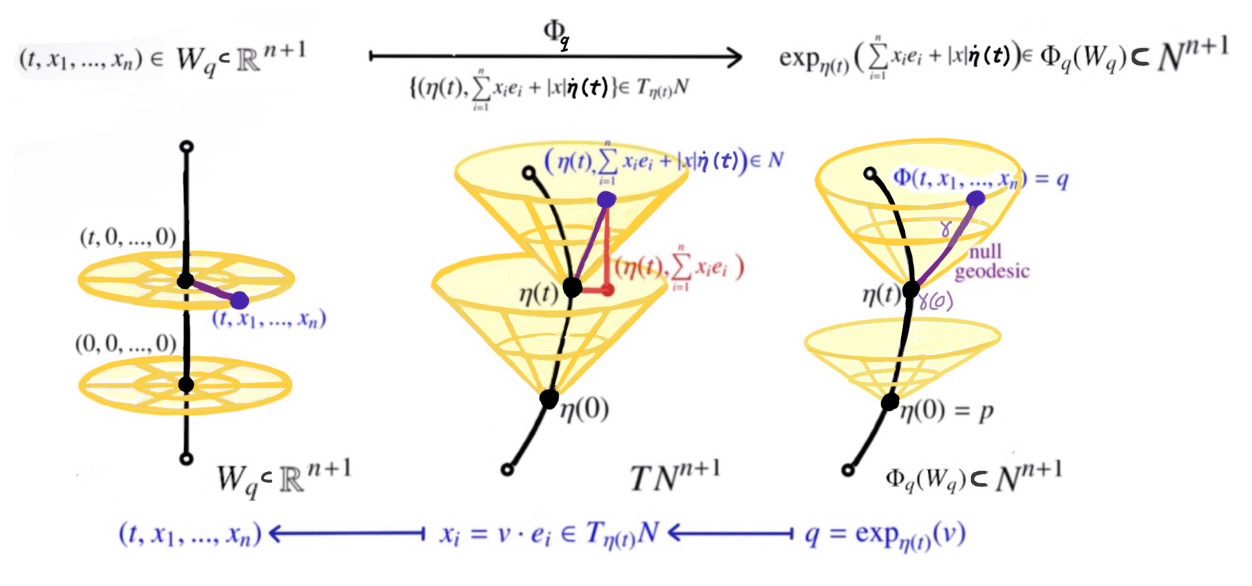} 
    \caption{ Here we see Temple's chart, $\Phi_{q}=\Phi_{q,\eta}:W_q\subset {\mathbb R}^{n+1} \to \Phi_{q}(W_q)\subset N$ as described in Theorem~\ref{thm-opt}.}
    \label{fig:Temple-precise}
\end{figure}

A key advantage of this chart is that the optical function $\omega_q$ can be used as an indicator of the causal future of a point $q$ in the following sense: for any $q' \in \Phi_q(W_q)$ we have 
\be
    \label{eq:IndicatorCausalFuture}
    \omega_q(q') \ge 0 \Leftrightarrow q' \in J^+(q,\Phi_q(W_q)),
\ee
see \cite[Lemma 3.6]{Sak-Sor-null}. 

The aim of this section is to prove Theorem \ref{thm:unif-Temple-intro} stated in the introduction (see also Theorem \ref{th-uniform-Temple} below) establishing the existence of the so called uniform Temple charts. Namely, we will show that for any $p \in N$ there is a neighborhood $U_p$, and a collection of Temple charts $\Phi_q: W_q \to V_q$ for $q \in U_p$, as described in Theorem \ref{thm-opt}, such that $U_p\subset V_q=\Phi_q(W_q)$. This result provides us with a collection of optical functions $\{\omega_q\}_{q\in U_p}$ that, roughly speaking, can be used to recover the causal future for all points $q\in U_p$ in the sense that \eqref{eq:IndicatorCausalFuture} holds for all $q,q'\in U_p$. In what follows, we call the neighborhood $U_p$ a \emph{uniform Temple neighborhood} and we call each $\Phi_q: W_q \to V_q$ a \emph{uniform Temple chart}. 

Although we will not emphasize it in the sequel, given a uniform Temple neighborhood $U_p$ and a collection of uniform Temple charts $\Phi_q:W_q \to V_q$ for $q\in U_p$, one can similarly recover causal pasts $J^-(q,\Phi_q(W_q))$ for all $q\in U_p$, as 
\be
    \label{eq:IndicatorCausalPast}
    \omega_q(q') \le 0 \Leftrightarrow q' \in J^-(q,\Phi_q(W_q)).
\ee

\subsection{Construction of Uniform Temple Charts}
\label{sec:construction}

Our construction of a uniform Temple neighborhood will be based on the following general result.

\begin{lem} 
    \label{lem-GeneralUniformity}
    Let $\Omega \subset N$ be an open set and let $S$ be a manifold such that $\dim(S) = \dim(N)$. Let $\psi: \Omega \times S \to N$ be a continuous map such that the maps $\psi_p: S \to N$ given by
    \be
        \psi_p(s) := \psi(p,s)
    \ee
    are injective. Assume further that there is $s_0 \in S$ such that $\psi_p(s_0) = p$ holds for all $p \in \Omega$. Then for any $p_0 \in \Omega$ there is an open set $U \subset \Omega$ containing $p_0$ such that for all $q \in U$ we have
    \be
        \psi_q(S) \supset U.
    \ee
\end{lem}

\begin{proof}
    Consider the map $\Psi: \Omega \times S \to N\times N$ given by
    \be
        \Psi(p,s) = (p,\psi_p(s)) = (p,\psi(p,s)).
    \ee
    Clearly, $\Psi$ is continuous by its definition. Moreover, it is injective: indeed, $\Psi(p,s) = \Psi(p',s')$ implies
    \be
        (p,\psi_p(s)) = (p',\psi_{p'}(s'))
    \ee
    in which case $p' = p$ and since $\psi_p=\psi_{p'}: S \to N$ is injective, we also get  $s = s'$. Since $\Psi: \Omega \times S\to N \times N$ is a continuous injection and 
    \be
        \dim(\Omega \times S) = \dim(N) + \dim(S) = 2\dim(N) = \dim(N\times N),
    \ee
    Brouwers Invariance of Domain Theorem implies that $\Psi(\Omega \times S) \subset N\times N$ is open. 

    We now fix $p_0 \in \Omega$. Since we have $\psi_p(s_0) = p$ for all $p \in \Omega$ it follows that $(p_0,p_0) = \Psi(p_0,s_0) \in \Psi(\Omega \times S)$. Since the topology of $N\times N$ is generated by products of open sets we can choose an open set $U$ such that $p_0 \in U$ and $U \times U \subset \Psi(\Omega \times S)$. 

    Now given any two points $q,q' \in U$ we have $(q,q') \in U\times U$ and by our construction it follows that $(q,q') \in \Psi(\Omega \times S)$. Consequently, we have 
    \be
        (q,q') = \Psi(p,s) = (p,\psi_p(s)) \text{ for some } (p,s) \in \Omega \times S.
    \ee
    It follows that $p = q$ and $q' = \psi_q(s)$. In other words, for all $q,q' \in U$, there is $s \in S$ such that $q' = \psi_q(s)$ and so for all $q \in U$ we have the inclusion $\psi_q(S) \supset U$. 
\end{proof}

The key step in the proof of the main result of this section, Theorem \ref{th-uniform-Temple}, is to apply the above lemma with $S = W_r:=(-r,r) \times B^n(r)$ and $\Omega=F(W_r)$ for some $r>0$ independent of $q$ to be determined later, so that $\Omega$ is a domain foliated by timelike geodesics as described in Proposition \ref{prop:FoliationByGeodesics}, and with $\psi(q,t,\vec{x})=\Phi_q(t,\vec{x})$ such that $\Phi_q$ are Temple charts as in \eqref{eqTempleClassic} suitably adapted to the foliation of $\Omega$ by timelike geodesics. For this, we will need the following simple result. 

\begin{prop}
    \label{prop:TempleInjectivity}
    Suppose that $V \subset \tilde V_p$ is open and that $\{e_0,e_1,\ldots, e_n\}$ is a frame field on $\tilde V_p$ where $e_0$ is timelike and $e_1,\ldots, e_n$ are spacelike. For a compact set $K \subset V$ we set
    \be
        R_0 := R_\mathrm{N}(K,V,\{e_a\})/\sqrt{2},
    \ee
    where $R_\mathrm{N}(K,V,\{e_a\})$ is the normal radius of $K$ as defined in Lemma \ref{lem:EXPDiffeo}.
    Let $\eta: I \to K \subset V$ be a  smooth future directed timelike curve. Then for all $s,t \in I$ and  $\vec x, \vec y \in B^n(R_0)$ we have 
    \be 
        \exp_{\eta(t)}\left(|\vec x|e_0 + \sum_{i = 1}^n x_i e_i\right)
        = \exp_{\eta(s)}\left(|\vec y|e_0 + \sum_{i = 1}^n y_i e_i\right) \quad \Longleftrightarrow  \quad (s,\vec x) = (t,\vec{y}).
    \ee 
\end{prop}

\begin{proof}
    Under the assumptions of the proposition, suppose that $s,t \in I$ and $\vec x, \vec y \in B^n(R_0)$ are such that 
    \be \label{eq:TempleInConvex}
        \exp_{\eta(t)}\left(|\vec{x}|e_0 + \sum_{i = 1}^n x_i e_i\right) 
        = \exp_{\eta(s)}\left(|\vec{y}|e_0+\sum_{i = 1}^n y_i e_i\right).
    \ee
    We denote this common point by $q$ and note that by Lemma \ref{lem:EXPDiffeo} we have $q\in V$ since
    \be
        \big|(|\vec x|, x_1 ,\ldots, x_n )\big| 
        = \sqrt{2}|\vec x| < \sqrt{2}R_0 =R_N(K,V, \{e_a\}).
    \ee
    For the rest of the proof, we restrict to the Lorentzian manifold $(\tilde V_p,g)$.
    
    First, we will show that $t = s$. Suppose, on the contrary, that this is not the case, so we may without loss of generality assume that $t > s$. On the one hand, \eqref{eq:TempleInConvex} implies that $q$ and $\eta(s)$ are connected by a \emph{null} geodesic contained in $(\tilde V_p,g)$.
    
    On the other hand, we see that $q\in J^+(\eta(t))$ while $\eta(t)\in I^+(\eta(s))$. Thus, by the push up principle for the spacetime \( \tilde V_p \), see for example \cite[Theorem 2.24]{Minguzzi-2019}, we conclude that $q \in I^+(\eta(t))$ within $\tilde V_p$. Then, by \cite[Corollary $2.10$]{Minguzzi-2019}, there is a \emph{timelike} geodesic from $\eta(s)$ to $q$ in $\tilde V_p$. Since $\tilde V_p$ is a normal convex neighborhood we get a contradiction:  $\eta(s)\in\tilde V_p$ and $q\in\tilde V_p$ cannot be joined by both a null geodesic and a timelike geodesic contained in $\tilde V_p$. This proves  our claim that $s = t$. 

    Now that $s = t$ we see that we have two null geodesics joining $\eta(s) = \eta(t)$ and $q$. Again, by the properties of $\tilde V_p$, they must be the same geodesic, which yields $\vec x = \vec y$.
\end{proof}

We can now prove our theorem establishing the existence of uniform Temple charts. For a shorter summary of this result, see Theorem \ref{thm:unif-Temple-intro} in the introduction. 

\begin{thm} 
    \label{th-uniform-Temple}
    Let $p\in N$, and let $\tilde V_p$ be its convex normal neighborhood as described in Section \ref{subsec:ConvexNormal}. Then there exists a compact set $K_p$ such that 
    \be
        p \in K_p \subset \tilde V_p
    \ee
    and a frame field $\{e_a\}=\{e_0,\ldots ,e_n\}$ on $K_p$ where $e_0$ is timelike and $e_1,\ldots,e_n$ are spacelike. Furthermore, there exists $R=R(p) > 0$, such that for all $r \in (0,R]$ there is a neighborhood $U_p = U_{p,r}$ such that $p \in U_p \subset K_p$, with the following properties:
    \begin{enumerate}
        \item For every $q \in U_p$, the geodesic $\eta_q$ satisfying the initial conditions $\eta_q(0) = q$, $\dot \eta_q(0) = e_0$ is defined on the interval $(-r,r)$ and satisfies 
        \be
            \eta_q(t) \in K_p \quad \text{and} \quad \dot\eta_q(t) = e_0 \quad \text{for all $t \in (-r,r)$.}
        \ee 
        \item For all $q \in U_p$, the Temple chart $\Phi_q = \Phi_{\eta_q}: W_r := (-r,r) \times B^n(r) \to \Phi_q(W_r)$ defined by
        \be
            \label{eq:TempleAgain}
            \Phi_q(t,\vec x) = \exp_{\eta_q(t)}\left(|\vec x|e_0 + \sum_{i = 1}^n x^i e_i\right),
        \ee
        is a homeomorphism and it is also a diffeomorphism away from the central axis $\{(t,\vec x) \in \mathbb R^{n+1} \,:\, |\vec x| = 0\,\}\cap W_r$.
        \item For all $q \in U_p$ we have
        \be
            U_p \subset \Phi_q(W_r)\subset K_p.
        \ee
        \item The statements (1)-(3) remain valid if we replace $U_p$ by any open set $\widetilde U_p$ such that $p\in \widetilde U_p \subset U_p$.
    \end{enumerate}
\end{thm} 

\begin{proof}
    Given $p \in N$ we let $R_p>0$, the frame field $\{e_a\}$ and the diffeomorphism $F: W_{R_p} \to V$ be as in the conclusion of Proposition \ref{prop:FoliationByGeodesics}. Next, we define the compact set
    \be
        K_p := F(\overline {W_{3R_p/4}}),
    \ee  
    in which case $\{e_a\}$ is defined on all of $K_p$. We then proceed to choose the radius $r > 0$ for the cylinders $W_r:=(-r,r)\times B^n(r)$ on which our uniform Temple charts will be defined. For this, we choose
    \begin{align}
        \epsilon_0 & := \epsilon(F(\overline {W_{3R_p/4}}),\{e_a\}) &&  \text{as in Lemma \ref{lem:JacobiTime}}, \\
        \delta_0 & := \delta(F(\overline{W_{R_p/2}}),F(W_{3R_p/4}),\{e_a\},\epsilon_0) && \text{as in Lemma \ref{lem:GoodGeodesics}}, \\  R_N &:= R_N(F(\overline {W_{R_p/2}}),F(W_{3R_p/4}),\{e_a\}) && \text{as in Lemma \ref{lem:EXPDiffeo}},
    \end{align}
    and we set
    \be 
        \label{eq:TempleRadiusDef}
        R = R(p) := \min\left(\delta_0,\frac{R_p}{4},\frac{R_N}{\sqrt{2}}\right).
    \ee
    Our goal now is to show  that the statements (1)-(3) hold for any $r \in (0,R]$.
    
    We fix an $r \in (0,R]$ for $R>0$ given by \eqref{eq:TempleRadiusDef}. Since $0 < r \leq \frac{R_p}{4}$, it is straightforward to check that
    every geodesic $\eta_q$ such that $\eta_q(0) = q \in F(W_r)$ and $\dot \eta_q(0) = e_0$ is defined on the interval $(-r,r)$ and satisfies
    \be
        \eta_q(t) = F(t_q + t,\vec x_q), \, \text{ where $(t_q,\vec x_q) \in W_r$ is such that $F(t_q,\vec x_q) = q$}.
    \ee
    In particular, we have $\eta_q: (-r,r)\to F(W_{R_p/2}) \subset K_p$. This shows that the first claim of the theorem holds
    for any open set $U_p$ such that $U_p \subset F(W_r)$. 
    
    Next, we will show that the Temple charts $\Phi_q=\Phi_{q,\eta_q}: W_r \to \Phi_q(W_r)$ have the properties described in the second claim of the proposition, as long as $q\in U_p$ where $U_p$ is any open subset of $F(W_r)$ .   
    For this, we define the cone 
    \be
        \Lambda := B^{n+1}\left(\sqrt{2}r\right) \cap \{(|\vec x|, \vec x): \, \vec x \in \mathbb R^n\}\subset \mathbb R^{n+1},
    \ee
    and the projection map 
    \be
        \pi: \Lambda \to B^n(r) \subset \mathbb R^n , \quad  (s_0,s_1,...,s_n) \mapsto  (s_1,...,s_n).
    \ee
    This allows us to view a Temple chart $\Phi_q$ given by \eqref{eq:TempleAgain} as the composition
    \be 
        \label{temple=Ft-proj-inv}
        \Phi_q(t, \vec x) = (\EXP_{\eta_q(t)} \circ \pi^{-1})(\vec x).
    \ee
    We note that $\pi: \Lambda \to B^n(r)$ is a homeomorphism with the inverse $\pi^{-1}(\vec x) = (|\vec x|,\vec x)$, which is also a diffeomorphism away from the vertex of the cone, $(0,\vec 0) \in \mathbb R^{n+1}$. Moreover, since $\eta_q(t) \in F(\overline{W_{R_p/2}})$ for all $q\in F(W_r)$  and $t \in (-r,r)$, and since $r$ was chosen so that
    \be
        0 < \sqrt{2} r \leq R_N(F(\overline{W_{R_p/2}}),F(W_{3R_p/4}),\{e_a\}),
    \ee
    by Lemma \ref{lem:EXPDiffeo} we know that for every $t \in (-r,r)$ and $q\in F(W_r)$, the map 
    \be
        \EXP_{\eta_q(t)}: B^{n+1}(\sqrt{2}r) \to F(W_{3R_p/4})
    \ee
    is a diffeomorphism onto its image. All in all, it follows from  \eqref{temple=Ft-proj-inv}  that $\Phi_q:  W_r \to F(W_{3R_p/4}) \subset K_p$ is well defined and is continuous. Moreover, Proposition \ref{prop:TempleInjectivity} together with $0<r\leq \frac{R_N}{\sqrt{2}} $ implies that for all $q \in F(W_r)$ the map $\Phi_q: W_r \to \Phi_q(W_r)$ is injective and hence invertible. Thus $\Phi_q: W_r \to N$ is a continuous injection between two manifolds of the same dimension, hence, by Brouwers theorem of invariance of domain, it is a homeomorphism with an open image. Next, we will show that each $\Phi_q: W_r \to \Phi_q(W_r)$ is a local diffeomorphism away from the central axis  $\{(t,\vec x) \in \mathbb R^{n+1} \,:\, |\vec x| = 0\,\}\cap W_r$.  Since 
    \be
        \pi^{-1}: B^ n(r) \setminus \{\vec 0 \} \to \Lambda \setminus \{(0,\vec 0) \}
    \ee
    is a diffeomorphism, the vectors $\pi^{-1}_*\partial_{x_i}$ are linearly independent and tangent to $\Lambda$ away from the vertex. Since
    \be
        \EXP_{\eta_q(t)}: B^{n+1}(\sqrt{2}r) \to F(W_{3R_p/4})
    \ee
    is a diffeomorphism for any fixed $t \in (-r,r)$ and $q \in F(W_{r})$, it follows that the vectors
    \be
        (\Phi_q)_*\partial_{x_i} = (\EXP_{\eta_q(t)} \circ \pi^{-1})_*\partial_{x_i}, \quad i = 1,\ldots, n,
    \ee
    are linearly independent and tangent to the null cone emanating from $\eta_q(t)$, away from its vertex. Due to the inequality 
    \be
        r \leq \delta_0 = \delta(F(\overline{W_{R_p/2}}),F(W_{3R_p/4}),\{e_a\},\epsilon_0),
    \ee
    Lemma \ref{lem:GoodGeodesics} implies that every geodesic of the form $\gamma(\lambda) := \Phi_q(t,\lambda \vec x)$ with $(t,\vec x) \in W_r \setminus \{\vec x = 0\}$ satisfies 
    \be
        \max_{0 \leq a \leq n} \sup_{\lambda \in [0,1]}|g(\dot \gamma(\lambda),e_a)| < \epsilon_0.
    \ee
    In turn, our choice of $\epsilon_0 = \epsilon(F(\overline {W_{3R_p/4}}),\{e_a\})$ combined with Lemma \ref{lem:JacobiTime} guarantees that the Jacobi fields $(\Phi_q)_*\partial_t$ of geodesics  $\gamma(\lambda) = \Phi_q(t,\lambda\vec x) \in F(W_{3R_p/4})$, $\lambda \in [0,1]$, are timelike for  $(t,\vec x) \in W_r \setminus \{\vec x = 0\}$. Thus $(\Phi_q)_*\partial_t$ is linearly independent of the tangent vectors to a null cone and so
    \be
        \{(\Phi_q)_*\partial_t, (\Phi_q)_*\partial_{x_1},\dots,  (\Phi_q)_*\partial_{x_n}\}
    \ee
    are linearly independent at any $(t,\vec x)\in W_r \setminus \{\vec x = 0\}$. Consequently, the Temple chart
    \be 
        \label{eq:GlobalDiffeoTemple}
        \Phi_q: W_r \setminus \{\vec x = 0\} \to F(W_{3R_p/4})
    \ee
    is a local diffeomorphism. Since $\Phi_q: W_r \to \Phi_q(W_r)$ is a homeomorphism it follows that the map in \eqref{eq:GlobalDiffeoTemple} is a diffeomorphism onto its image for every $q \in F(W_r)$. This concludes the proof of the second claim of the theorem  for any $r\in (0,R]$ where $R>0$ as in \eqref{eq:TempleRadiusDef} and any open set $U_p$ such that $U_p \subset F(W_r)$.
    
    Finally, we will apply Lemma \ref{lem-GeneralUniformity} in order to find $U_p \subset F(W_r)$ such that $U_p \subset \Phi_q(W_r) \subset K_p$, completing the proof of the theorem. We let $\Omega = F(W_r)$, $S = W_r$ and define $\psi: \Omega \times S = F(W_r) \times W_r \to N$ by
    \be
       \psi: \Omega \times S = F(W_r) \times W_r \to N, \quad (q,t,\vec x) \mapsto \Phi_q(t,\vec x).
    \ee
    The continuity of $\psi$ follows from the fact that
    \be
        \psi(q,t,\vec x) =\Phi_{q}(t,\vec x) = \pi_2\left(\EXP^{\{e_a\}}\left(\pi_2(\EXP^{\{e_a\}}(q,t,\vec 0)),|\vec x|, \vec x\right)\right)
    \ee
    where both \be\pi_2: N\times N\to N, \quad (p,q) \mapsto q\ee and $\EXP^{\{e_a\}}$ are continuous. Moreover, we have $\psi(q,0,\vec 0) = q$ and, as explained above, our choice of $r>0$ implies that for all $q \in F(W_r)$ the maps $\psi_q = \Phi_q : W_r \to N$ are injective. Consequently, we may apply Lemma \ref{lem-GeneralUniformity} with $\Omega$, $S$ and $\psi$ as defined above to conclude that there is an open set $U_{p} \subset \Omega = F(W_r) \subset K_p$ such that $p \in U_{p}$ and for all $q \in U_{p}$ we have 
    \be
        U_{p} \subset \Phi_q(W_r) \subset K_p.
    \ee 
    The claim (4) of the theorem follows directly from our construction.
\end{proof}

\subsection{Uniform Gradient Estimates for Optical Functions
} 
\label{sec:estimate} 

Given a point $p \in N$, we  define the compact set $K_p$ and the frame field $\{e_a\}$ on (a neighborhood of) $K_p$ as in Theorem \ref{th-uniform-Temple} (see also Proposition \ref{prop:FoliationByGeodesics}). Note that in this case we have the Riemannian metric $g_R$ defined on a neighborhood of $K_p$ by \eqref{eq:Riemannization}, so we may equip $K_p$ with the corresponding distance function $d_{g_R}$ as described Section \ref{sec:notations}. We also define the radius $r = r(p)$ and a uniform Temple neighborhood $U_p=U_{p,r}$ of $p$ as in Theorem \ref{th-uniform-Temple}. Then for every $q\in U_p$ there is a uniform Temple chart $\Phi_q$ defined on the cylinder $W_r=(-r,r)\times B_r$ by 
\be
    \label{eq:Temple-for-gradient}
    \Phi_q(t,\vec x) = \exp_{\eta_q(t)}\left(|\vec x|e_0 + \sum_{i = 1}^n x^i e_i\right),
\ee
and covering $U_p$. More specifically, we have $U_p\subset V_q=\Phi_q(W_r) \subset K_p$. For this chart, we define the optical function $\omega_q:V_q\to \mathbb{R}$ respectively the radial function $\lambda_q:V_q\to \mathbb{R}$ by 
\be 
    \label{eq:optical-and-radial}
    \omega_q(\Phi_q(t,\vec{x})) = t \quad \text{ respectively } \quad \lambda_q(\Phi_q(t, \vec x) ) = |\vec x|,
\ee
see Theorem \ref{thm-opt}. We recall that $\omega_q$ and $\lambda_q$ are smooth in $V_q\setminus \eta_q$, where $\eta_q(t) = \Phi_q(t,\vec 0)$, $t\in (-r,r)$, is the central geodesic of the chart.

The goal of this section is to prove the following result: 
\begin{prop} 
    \label{prop-g_R}
    There exists a constant $C > 0$ that may only depend on the Lorentzian metric $g$ and the frame field $\{e_a\}$ on the compact set $K_p$, on the dimension of $(N,g)$, and on a positive constant $R$ such that $0<r\leq R$, such that for all $q \in U_p$ the optical function $\omega_q$ of the uniform Temple chart $\Phi_q: W_r\to V_q$ satisfies 
    \be 
        \label{eqEikonalLip}
        \left ||\nabla^{g_R} \omega_q|_{g_R} - \sqrt{2}\right| < C \lambda_q \quad \text{ in } \quad  V_q\setminus \eta_q.
    \ee
\end{prop}
We note that Sakovich and Sormani showed in \cite[Section III.B]{Sak-Sor-null} that for every $q$ in some neighborhood $U(p)$ of $p$ there exists a constant $C=C(q)$ such that \eqref{eqEikonalLip} holds with respect to the Riemannian metric $g^q_R$ given by
\be
    g^q_R(X,Y):=2g(X,J^q)g(Y,J^q)+g(X,Y),
\ee
where the vector field $J^q=(\Phi_q)_*\partial_t$ depends on the chart $\Phi_q$, and hence on the choice of $q$. In contrast, the estimate of Proposition \ref{prop-g_R} holds with both $C$ and $g_R$ that are independent of $q\in U_p$. 

The proof of Proposition \ref{prop-g_R} requires a few preliminary results, some of which can be traced back to Temple's original work  \cite{Temple-1938}. However, here we restate these results in modern notation and provide complete proofs. We also make sure that all estimates are independent of a particular choice of  $q\in U_p$. 

We start by explaining the notation that we will use. For a fixed $q\in U_p$ we will assume given a uniform Temple chart $\Phi_q:(t,\vec x) \mapsto \Phi_q(t,\vec x)$ defined on $W_r=(-r,r)\times B_r$ by \eqref{eq:Temple-for-gradient}, with the image denoted by $V_q=\Phi_q(W_r)$. For this chart, we let  $\partial_t$ denote the vector field on $V_q$ defined at every point $\Phi_q(t,x_1,\ldots, x_n)$ by $\frac{\partial}{\partial s} _{|_{s=t}} \Phi_q(s,x_1,\ldots, x_n)$. We note that the vector field $\partial_t$ is defined with respect to the given Temple chart $\Phi_q:W_r\to V_q$ and, as such, it depends on $q \in U_p$. However, we have chosen not to emphasize this dependence in order to avoid excessive notation. In a similar vein, we will sometimes suppress $q$ and denote the radial function $\lambda_q$  as in  \eqref{eq:optical-and-radial} by $\lambda$. 
 
Next, using the radial function $\lambda$ we define the functions 
\be
    u^i = \frac{x^i}{\lambda}, \quad i = 1,\ldots, n,
\ee
so that $\vec u = (u^1,\ldots, u^n)$ is a unit vector in $\mathbb R^n$. We let $\partial_\lambda$ be the vector field on $V_q\setminus \eta_q$ defined at $\Phi_q(t,\vec{x})=\Phi_q(t,\lambda \vec{u})$ with $|\vec x|\neq 0$, by $\dot\gamma_{(t,\vec{u})}(\lambda)$, where 
\be 
    \label{eq:null-geods}
    \gamma_{(t,\vec u)}(\lambda) = \Phi_q(t,\lambda u_1,\ldots, \lambda u_n) =\exp_{\eta_q(t)}\left(\lambda \left(e_0 + \sum_{i=1}^n u_i e_i \right) \right), \quad \lambda\in [0,r),
\ee
is a null geodesic. Again, we will suppress the dependence of the vector field $\partial_\lambda$ on $q$ in our notation, whenever we work within a fixed uniform Temple chart $\Phi_q:W_r\to V_q$.

Finally, we note that the vector field $\partial_t$ defined as above is the Jacobi field $J_{(t,\vec u)}$ of the geodesic variation $\{\gamma_{(t,\vec u)}\}_t$, where the unit vector $\vec u \in \mathbb R^n$ is fixed and $t$ varies. Whenever it causes no confusion, we will abbreviate the notation and write $J = J(\lambda)$ respectively $\gamma = \gamma(\lambda)$ instead of $J_{(t,\vec u)} = J_{(t,\vec u)}(\lambda)$ respectively $\gamma_{(t,\vec u)} = \gamma_{(t,\vec u)}(\lambda)$.

For the rest of this section, it will be assumed that a constant $C>0$ may only depend on the Lorentzian metric $g$ and the frame field $\{e_a\}$ restricted to the compact set $K_p$, on the dimension of $(N,g)$, and on  $R$ such that $0<r\leq R$. We will also use the notation $f=O(\lambda^\beta)$ for $\beta\in \mathbb{R}$ to indicate that $|f|\leq C\lambda^\beta$ for a constant $C>0$ as described above.

The proof of Proposition \ref{prop-g_R} will be based on the following three lemmas describing the behavior of the vector fields $\partial_t$ and $\partial_\lambda$ within their uniform Temple chart.  

\begin{lem} 
    \label{lem:components-1}
    Let $U_p$ be a uniform Temple neighborhood of $p\in N$ as in Theorem \ref{th-uniform-Temple}. Then, for any $q\in U_p$, the vector fields $\partial_t$ and $\partial_\lambda$ of the  uniform Temple chart $\Phi_q:W_r\to V_q$ covering $U_p$ satisfy
    \be 
        \label{eq:g-components-1}
        g(\partial_\lambda, \partial_\lambda) = 0  \quad \text{ and } \quad g(\partial_t, \partial_\lambda) = -1 \quad \text{ on } V_q \setminus \eta_q.
    \ee
\end{lem}

\begin{proof}
    The vector field $\partial_\lambda$ is null, hence $g(\partial_\lambda, \partial_\lambda) =  0$. For the proof of the second identity in \eqref{eq:g-components-1}, we note that $\partial_t$ is the Jacobi field $ J = J(\lambda)$ along the null geodesic $\gamma = \gamma(\lambda)$ satisfying the initial conditions
    \be
        J(0)  = e_0,  \qquad D_\lambda J(0) = 0, 
    \ee
    where $D_\lambda$ denotes the covariant derivative of the Lorentzian metric $g$ along the geodesic $\gamma$. Applying standard theory for Jacobi fields (see e.g. Do Carmo \cite[Chapter 5, Proposition 3.6]{DoCarmo})  
    we see that $g(J,\dot\gamma)$ is constant along $\gamma$. It follows that
    \be
        g(J(\lambda),\dot\gamma(\lambda)) = g(J(0),\dot\gamma(0)) = g\left(e_0,e_0 + \sum_{i = 1}^n u_i e_i\right) = -1,
    \ee
    hence we have
    \be
        g(\partial_t, \partial_\lambda) = g(J,\dot\gamma) = -1 \quad \text{ on } V_q\setminus \eta_q
    \ee
    as claimed.
\end{proof}

\begin{lem}
    \label{lem:parallelVF} 
    Let $U_p$ be a uniform Temple neighborhood of $p\in N$ as in Theorem \ref{th-uniform-Temple}. Then, for any $q \in U_p$ the vector field $\partial_\lambda$ of the uniform Temple chart $\Phi_q:W_r\to V_q$ covering $U_p$  satisfies 
    \be 
        \label{eq-g_lambda_0}
        g(\partial_\lambda, e_0) = -1+ O(\lambda) \quad \text{and} \quad g(\partial_\lambda,e_i) = u_i  + O(\lambda) \quad \text{for} \quad i = 1,\ldots, n,
    \ee
    on $V_q\setminus \eta_q$, where $\lambda=\lambda_q$ is the radial function of the chart. 
\end{lem}
\begin{proof} 
    The proof is very similar to that of Lemma \ref{lem:JacobiTime}. Given the frame field $\{e_a\}$ on $K_p$, our goal is to estimate the coefficients $h_a=h_a(\lambda)$ in the decomposition $\partial_\lambda =\sum_{a=0}^n h_a e_a$. Since $\partial_\lambda= \dot \gamma$ is parallel along $\gamma$, we have 
    \be
        \label{eq-Cov-Der}
        0 = D_\lambda \dot \gamma 
        = \sum_{a=0}^n h'_a e_a + \sum_{a,b,c = 0}^n h_a h_b \Gamma_{ab}^ce_c,
    \ee
    where the Christoffel symbols $\Gamma_{ab}^c$ are defined by $\Gamma_{ab}^ce_c = \nabla^g_{e_a}e_b$. Working on the compact set $K_p$, we are in a position to say that there is a constant $C > 0$ such that $|\Gamma_{ab}^c| < C$ (see our conventions for $C$ above). As a consequence, \eqref{eq-Cov-Der} yields
    \be
        \label{eq-par-covar}
        |h'_a| \leq C\sum_{b,c=0}^n |h_b h_c| \quad \text{for} \quad a=0,\ldots, n.
    \ee Next, we define $h = h(\lambda)$ by $h=g_R(\dot\gamma,\dot\gamma) = \sum_{a = 0}^n (h_a)^2$, where $g_R$ is the Riemannianization of the metric $g$ defined by \eqref{eq:Riemannization} using the frame field $\{e_a\}$. Then  $h' = \sum_{a=0}^n 2 h_a h'_a $, which in combination with \eqref{eq-par-covar} yields $|h'|\leq C h^{3/2}$ for $C>0$ as described above. Integrating and using $h(0)=2$, we obtain $h=2 + O(\lambda)$. With this estimate at hand, using the fact that $\dot\gamma(0) = e_0 + \sum_{i = 1}^n u^ie_i$,   we find from \eqref{eq-par-covar} that 
    \be
        h_0 = 1 + O(\lambda) \quad \text{and}  \quad h_i = u^i + O(\lambda) \quad \text{for $i = 1,\ldots, n$}.
    \ee
    The claim \eqref{eq-g_lambda_0} follows, recalling that $g(\partial_\lambda, e_0) = - g_R(\partial_\lambda, e_0) = - h_0$ and $g(\partial_\lambda, e_i) =  g_R(\partial_\lambda, e_i) =  h_i$.
\end{proof}

\begin{lem}
    \label{lem:components-2} 
    Let $U_p$ be a uniform Temple neighborhood of $p\in N$ as in Theorem \ref{th-uniform-Temple}. Then, for any $q \in U_p$ the coordinate vector field $\partial_t$ of the uniform Temple chart $\Phi_q: W_r \to V_q$ satisfies 
    \be \label{eq:JacobiField}
        g(\partial_t, \partial_t) = -1 + O(\lambda)  \quad \text{ on } V_q\setminus \eta_q
    \ee
    where $\lambda=\lambda_q$ is the radial function of the chart.
    More specifically, we have 
    \be \label{eqJCoeffs}
        \partial_t = (1 + O(\lambda_q))e_0 + \sum_{i = 1}^n O(\lambda_q) e_i  \quad \text{ on } V_q\setminus \eta_q.
    \ee   
\end{lem}

\begin{proof}
    Given $z=\Phi_q(t,\vec{x})\in U_p$, we will apply Lemma \ref{lem:JacobiTime} with the compact set $K_p$ and the geodesic $\tilde \gamma: [0,1] \to K_p$ defined by
    \be
        \tilde \gamma(s) := \gamma(\lambda s)\, \text{ where } \gamma=\gamma_{(t,\vec x/|\vec x|)} \text{ is as in \eqref{eq:null-geods} and } \lambda=|\vec x|.
    \ee
    Clearly, $\tilde \gamma'(s) = \lambda  \gamma'(\lambda s)$ hence by \eqref{eq-g_lambda_0}, at $z=\Phi_q(t,\vec x)\in U_p$, we get
    \be
        |g(e_a,\tilde \gamma')| = \lambda |g(e_a,\partial_\lambda)| = O(\lambda)  
        \quad \text{for all $a = 0,\dots, n$}.
    \ee 
    The result follows from \eqref{eq:JacobiEstimate} in Lemma  \ref{lem:JacobiTime} applied to $\tilde \gamma = \tilde \gamma (s)$ with $\epsilon = C \lambda$ for $C>0$ as described in the beginning of this section.
\end{proof}

We now have all the ingredients ready to prove the main result of this section.

\begin{proof}[Proof of Proposition \ref{prop-g_R}]
    Given $p\in N$, let $K_p$, $\{e_a\}$, and $R=R(p)$ be as in Theorem \ref{th-uniform-Temple}. Given any $r\in (0,R]$, let $U_p=U_{p,r}$ be a uniform Temple neighborhood of $p$ and let  $\Phi_q:W_r\to V_q$ be a uniform Temple chart with the optical function $\omega_q: V_q \to \mathbb{R}$. We recall that the frame field $\{e_a\}$ used to define the Riemannian metric $g_R$ on $K_p$  by  \eqref{eq:Riemannization} is orthonormal with respect to $g_R$. Consequently, the gradient of the optical function $\omega_q$  with respect to $g_R$, as defined in Section \ref{sec:notations}, can be written as 
    \be
        \nabla^{g_R} \omega_q = e_0 (\omega_q) e_0 + e_1 (\omega_q) e_1 + \ldots + e_n (\omega_q) e_n, 
    \ee
    and we have
    \be
       |\nabla^{g_R}\omega_q|_{g_R}^2 = g_R(\nabla^{g_R} \omega_q, \nabla^{g_R} \omega_q) = (e_0(\omega_q))^2 + (e_1(\omega_q))^2 + \ldots + (e_n(\omega_q))^2.
    \ee 
    Since $\omega_q$ satisfies the eikonal equation $g(\nabla^g \omega_q, \nabla^g \omega_q) = 0$ on $V_q\setminus \eta_q$ we have 
    \be
        (e_0(\omega_q))^2 = (e_1(\omega_q))^2 + \ldots + (e_n(\omega_q))^2.
    \ee
    We thereby obtain
    \be
        \label{eq:GradientReformulation}
         |\nabla^{g_R}\omega_q|_{g_R}^2 
         = 2(e_0 (\omega_q))^2. 
    \ee
    In turn, applying  \eqref{eqJCoeffs} we get
    \be
    \begin{split}
        e_0(\omega_q) 
       & = \frac{1}{1 + O(\lambda_q)}\left( \partial_t(\omega_q) - \sum_{i = 1}^n O(\lambda_q)e_i(\omega_q) \right)\\
        & = \partial_t (\omega_q)(1 + O(\lambda_q)) + \sum_{i = 1}^n O(\lambda_q)e_i(\omega_q).
        \end{split}
    \ee
    Since the optical function $\omega_q$ is just the coordinate function $t$ of the Temple chart $\Phi_q:W_r\to V_q$, we have $\partial_t (\omega_q) = 1$. We also  note that
    \be
        |e_i(\omega_q)| = |g_R(\nabla^{g_R} \omega_q, e_i)| \leq |\nabla^{g_R} \omega_q|_{g_R} |e_i|_{g_R} = |\nabla^{g_R} \omega_q|_{g_R}, \text{ for } i=1,\ldots,n.  
    \ee
     Consequently, we have
    \be
        e_0(\omega_q) 
        = 1 + O(\lambda_q)+ O(\lambda_q)|\nabla^{g_R}\omega_q|_{g_R}.
    \ee
    Combining this with \eqref{eq:GradientReformulation} 
    we find that
    \be
        |\nabla^{g_R}\omega_q|_{g_R} = \sqrt{2}+ O(\lambda_q) + O(\lambda_q)|\nabla^{g_R}\omega_q|_{g_R},
    \ee
    which yields $|\nabla^{g_R}\omega_q|_{g_R} = \sqrt{2} + O(\lambda_q)$ as claimed.
\end{proof}

In conclusion, we note the following corollary that will be useful in the next section where we study applications of uniform Temple charts.

\begin{cor} 
    \label{omega-Lip} 
    Given $p\in N$ we can choose $R=R(p)$ in Theorem \ref{th-uniform-Temple} so that for all $0<r\leq R$ the uniform Temple neighborhood $U_p=U_{p,r}$ of $p$ as defined in Theorem \ref{th-uniform-Temple} has the following property: for any $q\in U_p$  the optical function $\omega_q$ of the uniform Temple chart $\Phi_q: W_r\to V_q$ satisfies the estimate
    \be \label{eqUniLipBound}
        \sup\left\{ \frac{|\omega_q(z)-\omega_q(z')|}{d_{g_R}(z,z')} \,: \, z\neq z' \in V_q \, \right\} < 2,
    \ee
    where $d_{g_R}$ is the distance defined with respect to the Riemannian metric $g_R$ on  $V_q$.
\end{cor}

\begin{proof} 
    Given $p\in N$, we first choose $R=R(p)$ as in Theorem \ref{th-uniform-Temple} and we let  $U_p=U_{p,r}$ denote a uniform Temple neighborhood of $p$ for  $r\in (0,R]$. Next, in the view of Proposition \ref{prop-g_R}, we note that we may redefine $R=R(p)$ in Theorem \ref{th-uniform-Temple} so that for all $r\in(0, R]$ and $q\in U_p$  the optical function $\omega_q$ of the uniform Temple chart $\Phi_q: W_r\to V_q$  satisfies $|\nabla^{g_R} \omega_q|_{g_R} \leq 2$. 
    
    With these preliminaries at hand, we may now argue as in the proof of \cite[Lemma 3.9]{Sak-Sor-null}. Given $z\neq z'\in V_q$, we let $\beta_i:[0,1]\to V_q$  denote a family of piecewise smooth curves from $\beta_i(0)=z$ to $\beta_i(1)=z'$ such that
    \be
        \lim_{i\to \infty} L_{g_R}(\beta_i) = d_{g_R}(z,z').
    \ee
    We can assume that each $\beta_i$ hits the central geodesic of the chart, $\eta_q$, at most finitely many times, so that the optical function $\omega_q$ is differentiable along $\beta_i$ away from those times.  As a consequence, we have
    \be
        \begin{split}
            |\omega_q(z)-\omega_q(z')| 
            & \leq \int_0^1 | \tfrac{d}{d\sigma} \omega_q(\beta_i(\sigma)) | \, d\sigma 
            \\
            &= \int_0^1 |g_R( \nabla^{g_R} \omega_q (\beta_i(\sigma)), \beta_i'(\sigma))| \, d\sigma \\
            & \le \int_0^1 |\nabla^{g_R} \omega_q (\beta_i(\sigma))|_{g_R} |\beta_i'(\sigma))|_{g_R} \, d\sigma \\
            & \le 2 \int_0^1 |\beta_i'(\sigma) |_{g_R} \,d\sigma 
            = 2 L_{g_R}(\beta_i).
        \end{split}
    \ee
    Taking the limit as $i\to \infty$ we obtain
    \be
        \frac{|\omega_q(z)-\omega_q(z')|}{d_{g_R}(z,z')} \le 2 \text{ for all } z\neq z' \in V_q.
    \ee
    which proves the result.      
\end{proof}

\section{Applications of Uniform Temple Charts}
\label{sec:application}
\subsection{Review of Time functions and Null distance} \label{sec:review}

In order to prepare the reader to applications of our result on existence of uniform Temple charts, in this section we recall the notion of time function and the associated definition of null distance.

Given a spacetime $(N,g)$, a \emph{generalized time function} is a real valued function, $\tau: N\to \mathbb{R}$, that is strictly increasing along all future directed causal curves. Continuous generalized time functions are called \emph{time functions}.  For example, any \emph{temporal} function, that is a smooth function $\tau: N\to \mathbb{R}$ such that $\nabla \tau (p)$ is timelike past pointing for all $p\in N$, is a time function. In general, there are time functions which are not smooth, as the following example shows.

\begin{example}
    \label{ex-cosmotime}
    The \emph{cosmological time function} was defined by Andersson, Galloway, and Howard in \cite{AGH} (see also Wald and Yip \cite{Wald-Yip}) as follows:
    \be 
        \tau_{g}(p)= \sup \{L_g(C)\,| \, \textrm{ future timelike } C:[0,1]\to N \text{ with } C(1)=p\}
    \ee
    where $L_g(C)$ is the \emph{Lorentzian length} of a $C^1$-curve $C$ defined by
    \be
        L_g(C)=\int_0^1 |g(C'(s),C'(s))|^{1/2}\,ds.
    \ee
    If $\tau_g(q)<\infty$ for all $q \in N$ and if $\tau_g \to 0$ on past inextendible causal curves then the cosmological time function $\tau_g$ is said to be \emph{regular}. In general, a regular cosmological time function is only Lipschitz with $g(\nabla \tau_g, \nabla \tau_g) = -1$ almost everywhere, see \cite{AGH}. 
\end{example}

Although some of the results that we review below hold in greater generality, for the purposes of this paper it will be convenient to work with the following class of time functions introduced by Burtscher and Garc\'ia-Heveling in \cite[Definition 1.6]{Burtscher-Garcia-Heveling-22}:

\begin{definition}
    \label{def:WeakTemporalFunction}
    We say that a time function $\tau:N \to \mathbb{R}$ is \emph{weak temporal} if for every point $p \in N$ there is a constant \( C > 1\) and a neighborhood $U$ of $p$ that has a Riemannian metric $g_U$ with a definite distance function $d_U: U \times U \to [0,\infty)$ such that we have
    \be
        \label{eq:WeakTemporal}
        \frac{1}{C}d_U(q,q') 
        \leq |\tau(q') - \tau(q) |
        \leq Cd_U(q,q') \quad \text{ for all } \quad (q,q') \in J^+ \cap (U\times U).
    \ee
\end{definition}

Examples of weak temporal functions include the aforementioned temporal functions and regular cosmological time functions as in Example \ref{ex-cosmotime}, see Burtscher and Garc\'ia-Heveling \cite[Lemma 2.12]{Burtscher-Garcia-Heveling-22}. We note that the first inequality in \eqref{eq:WeakTemporal} is the \emph{locally anti-Lipschitz condition} of Chru\'sciel, Grant and Minguzzi, see  \cite{CGM}, while the second inequality is referred to as \emph{locally Lipschitz condition}. Importantly, weak temporal functions are indeed locally Lipschitz in the usual sense, see Burtscher and Garc\'ia-Heveling \cite[Proposition 2.11]{Burtscher-Garcia-Heveling-22}.


Given a generalized time function $\tau: N  \to \mathbb{R}$, we can equip $(N,g)$ with \emph{null distance} as defined by Sormani and Vega in \cite{SV-null}:

\begin{definition} 
    \label{def:null}
    Let $\tau: N \to \mathbb{R}$ be a generalized time function on a spacetime $(N,g)$. Let $ \beta: [a,b] \to N$ be a piecewise causal curve with break points $x_i = \beta(s_i)$, $0 \leq i \leq m$, where each smooth causal segment may be either future pointing or past pointing. The \emph{null length of $\beta$} is defined by
    \be
        \hat{L}_\tau (\beta) = \sum_{i=1}^m |\tau(x_i)-\tau(x_{i-1})|.
    \ee
    The \emph{null distance} between $p,q\in N$ is defined by
    \be
        \dhat_\tau (p,q)=  \inf \{ \hat{L}_\tau(\beta) : \beta \textrm{ piecewise casual from } p \textrm{ to } q \textrm{ via } x_i \in \beta\}.
    \ee
\end{definition}

In general, $\dhat_\tau: N\times N \to \mathbb{R}$ is only a pseudometric for which the property 
\be 
    \label{eq:definite}
    \dhat_\tau(p,q)=0 \Rightarrow p=q
\ee
may fail. On the other hand, when  $\tau$ is a weak temporal function, the result of Sormani and Vega \cite[Theorem 4.6]{SV-null} implies that $(N,\dhat_\tau)$ is a definite metric space satisfying (\ref{eq:definite}) where $\dhat_\tau$ induces the manifold topology.

Sormani and Vega also proved that for any generalized time function $\tau$ we have 
\be
    \label{eq:encodes-causality}
    q \in J^+(p) \implies \hat{d}_\tau(p,q) = \tau(q) - \tau(p),
\ee
see \cite[Lemma 3.10]{SV-null}. They also showed that the converse is true on Minkowski space when equipped with the standard time function \( \tau(t,\vec x) = t \), where the $\hat{d}_\tau$-balls are cylinders that are perfectly aligned with null cones, and conjectured it was true more generally. Later, Sakovich and Sormani showed that while the converse does not hold true in general (see Examples 2.1-2.2 in \cite{Sak-Sor-null}), it is true for locally anti-Lipschitz generalized time functions with compact level sets (see \cite[Theorem 4.1]{Sak-Sor-null}). See also subsequent work of Burtscher and Garc\'ia-Heveling in \cite{Burtscher-Garcia-Heveling-22} and Galloway in \cite{Galloway-Null}, where this result was generalized. 

Sormani and Vega originally defined the null distance as a part of the program aiming to define weak convergence of spacetimes which do not converge smoothly, following a suggestion of Shing-Tung Yau and Lars Andersson. As described in \cite{Sormani-Oberwolfach-18}, the plan was to convert a sequence of spacetimes canonically into metric spaces using the cosmological time function and then take the intrinsic flat limit of the sequence. This approach was tested for warped product spacetimes with smooth cosmological time functions by Allen and Burtscher in \cite{Allen-Burtscher-20}. In a similar spirit, Allen and Burtscher \cite{Allen-Burtscher-20} and Allen  \cite{Allen-23} studied Gromov-Hausdorff convergence of warped products converted into definite metric spaces using their null distances, while Kunzinger and Steinbauer \cite{Kunzinger-Steinbauer-21} did similar investigations in the framework of Lorentzian length spaces. At the same time, Graf and Sormani \cite{Graf-Sormani} obtained estimates required for studying convergence of more general spacetimes. More recently, Sakovich and Sormani \cite{Sak-Sor-24} used the null distance to introduce several  notions of distances, along with associated notions of convergence, for broad classes of spacetimes $(N,g)$ equipped with regular cosmological time functions $\tau_g$ and associated null distances $\dhat_g=\dhat_{\tau_g}$. Some of these notions, most notably \emph{intrinsic timed-Hausdorff distance}, were proven to be definite using the causality encoding theorem of Galloway \cite[Theorem 3]{Galloway-Null} and the isometry theorem of Sakovich and Sormani \cite[Theorem 1.3]{Sak-Sor-null}. For more details on this, and for comparison with other notions of distances between spacetimes please see \cite{Sak-Sor-24}. For further results on intrinsic timed-Hausdorff distance, please see \cite{Che-Per-Sor}.

In the remaining sections we will prove a more general version of the aforementioned isometry theorem \cite[Theorem 1.3]{Sak-Sor-null}, by applying our uniform Temple charts constructed in Section \ref{sec:Temple}. We will also use these charts to show that the metric space $(N,\dhat_\tau)$ is countably rectifiable. These results will be applied in the upcoming work \cite{future-work} on intrinsic flat convergence of spacetimes.  

\subsection{Uniform Temple Charts are bi-Lipschitz}
\label{sec:biLip}
 In this section we prove Theorem \ref{thm:Lip-intro} which we restate here for the reader's convenience: 

\begin{thm}
    \label{thm:Lip}
    Let $(N,g)$ be a spacetime equipped with a weak temporal function \( \tau \) 
    and let $\dhat_\tau$ be the associated null distance. Given $p\in N$ we can choose its uniform Temple neighborhood $U_p$ so that for any $q\in U_p$ the restriction of the uniform Temple chart $\Phi_q: W_r \to \Phi_q(W_r)\supset U_p $ to $\Phi_q^{-1}(U_p)$ satisfies  
    \be
        \quad \Phi_q: (\Phi_q^{-1}(U_p), d_{\mathbb{E}^{n+1}}) \to (U_p,\dhat_\tau) \quad  \text{is bi-Lipschitz}.
    \ee
\end{thm}

We begin the proof with the following lemma.

\begin{lem}\label{lemmaSimple}
    Let $(N,g)$ be a spacetime equipped with a weak temporal function \( \tau \). 
    Given $p\in N$ we can choose its uniform Temple neighborhood $U_p$ so that the following holds: 
    \begin{itemize}
        \item[(1)] For all $q\in U_p$ the uniform Temple charts $\Phi_q: W_r \to V_q=\Phi_q(W_r)$ satisfy 
        \be\label{eq:Inclusion}
            U_p \subset V_q \subset K_p \subset U,
        \ee
        where $U$ is the neighborhood of $p$ as in Definition \ref{def:WeakTemporalFunction} and $K_p$ is the compact set equipped with the frame field $\{e_a\}$ as described in Theorem \ref{th-uniform-Temple}. 
        \item[(2)] Let $g_R$ denote the Riemannian metric defined by  \eqref{eq:Riemannization} on $K_p$ and let $d_{g_R}$ be the associated Riemannian distance. Then there is a constant $\widetilde{C}>1$ that may only depend on the Lorentzian metric $g$ and the frame field $\{e_a\}$ on the compact set $K_p$ as well as $g_U$ and $C$ as in Definition \ref{def:WeakTemporalFunction},   
        such that  we have
        \be
            \label{eqCGM2}
           \frac{1}{\widetilde{C}}\, d_{g_R}(q,q') \leq |\tau(q)-\tau(q')| \leq \widetilde{C}d_{g_R}(q,q') \,\,\,\, \text{for all} \,  q,q'\in J^+\cap (U_p\times U_p).
        \ee
    \end{itemize}
\end{lem}

\begin{proof}
    We recall from the proof of Theorem \ref{th-uniform-Temple} that $K_p$ is of the form $K_p=F(\overline{W_{3 r/4}})$,  where $F$ is a smooth map as defined in  Proposition \ref{prop:FoliationByGeodesics}, so it is clear that by choosing $r > 0$ to be sufficiently small we can ensure that $K_p \subset U $, where $U$ is an open set equipped with the Riemannian metric $g_U$ and the associated distance $d_U$ as in the formulation of this lemma. This implies \eqref{eq:Inclusion}, since the first two inclusions were proven in Theorem \ref{th-uniform-Temple}. 

    For proving the second claim of the lemma, we first recall that a weak temporal function $\tau:N\to\mathbb{R}$ satisfies \eqref{eq:WeakTemporal}. Due to the inclusion \( U_p \subset U \) we have
    \be
        \label{eq:WeakTemporal2}
        \frac{1}{C}d_U(q,q') 
        \leq |\tau(q') - \tau(q) |
        \leq Cd_U(q,q') \text{ for all } q,q' \in J^+ \cap (U_p\times U_p).
    \ee
    In order to promote this estimate to \eqref{eqCGM2} it suffices to note that the distances associated to the continuous Riemannian metrics $g_U$ and $g_R$, both defined on the compact set $K_p$, are equivalent, see e.g. \cite[Theorem 4.5]{Burtscher-12}. 
\end{proof}
 

Based on this lemma, we can prove  the following two propositions, that combined together will imply the result of Theorem \ref{thm:Lip}.

\begin{prop}
    \label{prop:Lip1}
    Let $(N,g)$ be a spacetime equipped with a weak temporal function \( \tau \). 
    Given $p\in N$ let  $U_p$ be its uniform Temple neighborhood as described in Lemma \ref{lemmaSimple}. 
    Then the identity map $\mathrm{id}: (U_p, d_{g_R}) \to (U_p,\dhat_\tau)$ is bi-Lipschitz. In particular, there is a constant \( C' > 1 \) depending only on the constant \(\widetilde{C}\) of Lemma \ref{lemmaSimple} 
    such that we have 
    \be 
        \label{eqDistComparison}
        \frac{1}{C'} d_{g_R} (q,q') 
        \leq \hat{d}_\tau (q,q') 
        \leq C' d_{g_R} (q,q') 
    \ee 
    for all $q,q'\in U_p$.
\end{prop}

\begin{proof} 
    The left part of the inequality \eqref{eqDistComparison} is essentially a consequence of the local anti-Lipschitz property of $\tau$\footnote{See Vega \cite[Lemma 3.34]{Vega21} for a similar global result in the case when $\tau$ is smooth.}. Indeed, given any piecewise causal curve $\beta$ from $q$ to $q'$ in $U_p$ with breaks at $\beta(t_i)$, $i=1,\ldots,N+1$, we may use \eqref{eqCGM2} and the triangle inequality to obtain 
    \be
        \begin{aligned}
            \hat{L}_\tau (\beta) 
            & = \sum_{i=1}^N |\tau(\beta(t_{i+1}))-\tau(\beta(t_i))| \\ 
            & \geq \frac{1}{\widetilde C} \sum_{i=1}^N d_{g_R}(\beta(t_{i+1}),\beta(t_i)) \\ 
            & \geq \frac{1}{\widetilde C} d_{g_R} (q,q'),
        \end{aligned}
    \ee
    where the constant $\widetilde C$ is as in \eqref{eqCGM2}.
    Taking the infimum over all piecewise causal curves $\beta$ from $q$ to $q'$ in $U_p$ we get the desired bound 
    \be
        \hat{d}_\tau (q,q')\geq \frac{1}{C'} d_{g_R} (q,q').
    \ee
The proof of the right part of  \eqref{eqDistComparison} is more involved, as it crucially uses properties of uniform Temple charts, in particular the uniform bound of Corollary \ref{omega-Lip}. More specifically, given any two points $q_1,q_2\in U_p$, let $\Phi_{q_1}= \Phi_{ \eta_{q_1}}:W_r\to V_{q_1}=\Phi_{q_1} (W_r)$ be a uniform Temple chart centered at $q_1$, as described in Theorem \ref{th-uniform-Temple}. Since $q_2\in U_p$ is covered by the image of this chart, there exists a unique point $(t_2,\vec{x}_2)\in W_r$ such that $q_2=\Phi_{q_1}(t_2,\vec{x}_2)$.  We set $q_{12}= \eta_{q_1}(t_2)$ and note that $q_1$ and $q_{12}$ are joined by a timelike unit speed geodesic $\eta_{q_1}$, such that $\dot\eta_{q_1} = e_0$ for all values of the arclength parameter. Furthermore, $q_{12} $ and $q_{2}$ are joined by the null geodesic $\gamma_{(t_2,\vec u_2)} (\lambda) = \Phi_{q_1} (t_2, \lambda \vec u_2)$  where $\vec u_2 = \vec x_2/ |\vec x_2|$ and $\lambda\in [0,|\vec x_2|]$.  We note also that by the definition of the optical function $\omega_{q_1}$ of the chart $\Phi_{q_1}: W_r \to \Phi_{q_1} (W_r)$ we have 
    \be
        \label{eq:using-optical}
        q_1 = \eta_{q_1}(0)=\eta_{q_1}(\omega_{q_1}(q_1)) \quad \text{ and } \quad 
        q_{12} = \eta_{q_1}(t_2)=\eta_{q_1}(\omega_{q_1}(q_2)).
    \ee
    Let $d_{g_R}$ denote the distance induced by the Riemannian metric $g_R$ on $K_p$. Then, for the described choice of $q_{12}$ we have 
    \[
        \begin{split}
        \dhat_\tau (q_1,q_2) 
        & \leq \dhat_\tau (q_1,q_{12}) + \dhat_\tau (q_{12},q_2) 
        \hspace{3cm} \textrm{ by the triangle inequality} \\ 
        & = | \tau (q_{1}) - \tau(q_{12})|+ |\tau (q_{12}) - \tau(q_2)|
        \hspace{0.6cm} \textrm{ $\eta_{q_1}$ is timelike and $q_2\in J^+(q_{12})$} \\
        & \leq \widetilde C (d_{g_R}(q_1, q_{12}) + d_{g_R}(q_{12}, q_2))
        \hspace{3.7cm} \textrm{ by Lemma \ref{lemmaSimple}}\\
        & \leq \widetilde C (2 d_{g_R}(q_1, q_{12})+ d_{g_R}(q_{1}, q_2)) 
        \hspace{2.0cm} \textrm{ by the triangle inequality} \\
        & = \widetilde C (2 d_{g_R}(\eta_{q_1}(\omega_{q_1}(q_1)), \eta_{q_1}(\omega_{q_1}(q_2))) + d_{g_R}(q_{1}, q_2)) 
        \hspace{1.5cm} \textrm{ by \eqref{eq:using-optical}} \\
        & \leq \widetilde C (2|\omega_{q_1}(q_1)-\omega_{q_1}(q_2)| + d_{g_R}(q_{1}, q_2)) 
        \hspace{1cm} \textrm{ since $|\dot\eta_{q_1}|_g=|\dot\eta_{q_1}|_{g_R}=1$} \\
        & \leq \widetilde C (4d_{g_R}(q_{1}, q_2) + d_{g_R}(q_{1}, q_2))
        \hspace{4.6cm} \textrm{ by \eqref{eqUniLipBound}} \\ &   = 5\widetilde C d_{g_R}(q_{1}, q_2).
        \end{split}
    \]
    Since $U_p \subset K_p$ this implies the right inequality of \eqref{eqDistComparison} and hence completes the proof.
\end{proof}

\begin{prop}
    \label{prop:Lip2}
    Let $(N,g)$ be a spacetime equipped with a weak temporal function \(\tau\). 
    Given $p\in N$ let  $U_p$  be its uniform Temple neighborhood, and let $\Phi_q: W_r \to V_q=\Phi_q(W_r) \supset U_p$ be a uniform Temple chart, both as constructed in Lemma \ref{lemmaSimple}. Then the restriction
    \be
        \label{eq:bi-Lip-sec6}
        \Phi_q: (\Phi_q^{-1}(U_p),d_{\mathbb{E}^{n+1}}) \to (U_p,d_{g_R}) 
    \ee 
    is bi-Lipschitz. In particular, there is a constant $c>1$ that may only depend on the Lorentzian metric $g$ and the frame field $\{e_a\}$ on the compact set $K_p$ (as in Lemma \ref{lemmaSimple}) 
    such that
    \be 
        \label{eqDistComparison2}
        \frac{1}{c} d_{\mathbb{E}^{n+1}} (\Phi_q^{-1}(z),\Phi_q^{-1}(z')) \leq d_{g_R}(z,z') \leq c d_{\mathbb{E}^{n+1}} (\Phi_q^{-1}(z),\Phi_q^{-1}(z')) 
    \ee 
    holds for all $z,z'\in U_p$.
\end{prop}
\begin{proof}
    In this proof, \( \vec x \), \( \vec x_1 \) and \( \vec x_2 \) denote vectors in \( \mathbb R^n \). We begin by recalling Lemma \ref{lem:framed-exp}, according to which the map 
    \be
        \label{eq:EXPEXP}
        \EXP: \EXP^{-1}(U_p\times U_p) \to U_p \times U_p
    \ee
    is a diffeomorphism, and hence it is also a bi-Lipschitz map. Now, using any of the Temple's charts $\Phi_q: W_r \to V_q=\Phi_q(W_r) \supset U_p$ for $q\in U_p$ as described above, we may express this map as 
    \be
        \EXP\left(\eta_q(t), |\vec x|, \vec x\right)=(\eta_q(t),\Phi_q(t,\vec x))
    \ee
    since
    \be
        \Phi_q(t,\vec x)=\EXP_{\eta_q(t)}\left(|\vec x|,\vec x\right).
    \ee
    The fact that the map \eqref{eq:EXPEXP} is bi-Lipschitz implies that there is a constant $K>1$ such that for any 
    \be
        w_1 = \left(\eta_{q}(t_1),|\vec x_1|, \vec x_{1} \right) \quad \text{and}\quad  
        w_2 = \left(\eta_{q}(t_2),|\vec x_2|, \vec x_{2} \right) 
    \ee
    in $\EXP^{-1}(U_p\times U_p)$ we have 
    \be 
        \label{eq:EXPisBiLip}
        \frac{1}{c} d_{g_R\times g_{\mathbb{E}^{n+1}}} (w_1,w_2) \leq d_{g_R\times g_R} (\EXP(w_1),\EXP(w_2)) \leq c d_{g_R\times g_{\mathbb{E}^{n+1}}} (w_1,w_2)
    \ee
    where 
    \be
        d_{g_R\times g_{\mathbb{E}^{n+1}}} (w_1,w_2)=\sqrt{d_{g_R}\left(\eta_q(t_1),\eta_q(t_2)\right)^2 + ||\vec x_1|-|\vec x_2||^2+ |\vec x_1-\vec x_2|^2}
    \ee
    and 
    \be
        d_{g_R\times g_R} (\EXP(w_1),\EXP(w_2))=\sqrt{d_{g_R}\left(\eta_q(t_1),\eta_q(t_2)\right)^2 + d_{g_R} \left(\Phi_q(t_1,\vec x_1), \Phi_q(t_2,\vec x_2)\right)^2}.
    \ee
    We will use \eqref{eq:EXPisBiLip} to prove \eqref{eqDistComparison2} as follows. 

    First, using the right part of the inequality in \eqref{eq:EXPisBiLip} as a starting point, after simple manipulations we get
    \be
        d_{g_R} \left(\Phi_q(t_1,\vec x_1), \Phi_q(t_2,\vec x_2)\right)^2 \leq c^2\left( d_{g_R}\left(\eta_q(t_1),\eta_q(t_2)\right)^2 + ||\vec x_1|-|\vec x_2||^2+ |\vec x_1-\vec x_2|^2\right).
    \ee
    Next we note that $||\vec x_1|-|\vec x_2||\leq |\vec x_1-\vec x_2|$, as a consequence of the triangle inequality, and that  
    \be
        \label{eq:UnitSpeed}
        d_{g_R}\left(\eta_q(t_1),\eta_q(t_2)\right) \leq |t_1-t_2|  
    \ee
    holds, since $|\dot \eta_q|_{g_R} = 1$. Consequently, we obtain
    \be
        d_{g_R} \left(\Phi_q(t_1,\vec x_1), \Phi_q(t_2,\vec x_2)\right) \leq \sqrt{2} c \sqrt{|t_1-t_2|^2 + |\vec x_1-\vec x_2|^2 }
    \ee
    which proves the right part of \eqref{eqDistComparison2}.

    In order to prove the left part of \eqref{eqDistComparison2} we first note that by our estimate for the optical function of the uniform Temple chart \eqref{eqUniLipBound}, we have 
    \be
        \label{eqConsequenceUniLipBound}
        d_{g_R} \left(\Phi_q(t_1,\vec x_1), \Phi_q(t_2,\vec x_2)\right) \geq \frac{1}{2} |t_1-t_2|.
    \ee
    On the other hand, using the left part of the inequality in \eqref{eq:EXPisBiLip}, it is straightforward to obtain
    \be
        |\vec x_1-\vec x_2|^2  \leq (c^2-1) d_{g_R}\left(\eta_q(t_1),\eta_q(t_2)\right)^2 + c^2 d_{g_R}\left(\Phi_q(t_1,\vec x_1), \Phi_q(t_2,\vec x_2)\right)^2 .
    \ee
    Combining \eqref{eq:UnitSpeed} and \eqref{eqConsequenceUniLipBound} we get
    \be
        d_{g_R}\left(\eta_q(t_1),\eta_q(t_2)\right)^2 \leq |t_1-t_2|^2 \leq 4 d_{g_R} \left(\Phi_q(t_1,\vec x_1), \Phi_q(t_2,\vec x_2)\right)^2
    \ee
    hence
    \be 
        \label{eq:XisLipschitz}
        |\vec x_1-\vec x_2|^2  \leq (5c^2-4)  d_{g_R}\left(\Phi_q(t_1,\vec x_1), \Phi_q(t_2,\vec x_2)\right)^2.
    \ee
    Ultimately, combining \eqref{eq:XisLipschitz} and \eqref{eqConsequenceUniLipBound} it is straightforward to check that
    \be
        d_{g_R} \left(\Phi_q(t_1,\vec x_1), \Phi_q(t_2,\vec x_2)\right) \geq \frac{1}{\sqrt{5} c} \sqrt{|t_1-t_2|^2 + |\vec x_1-\vec x_2|^2 }
    \ee
    which proves the left part of \eqref{eqDistComparison2}, completing the proof. 
\end{proof}

Finally, we are able to prove the main result of this section.

\begin{proof}[Proof of Theorem \ref{thm:Lip}] 
    Let $(N,g)$ be a spacetime as in the statement of the theorem. Given $p\in N$ we define its uniform Temple neighborhood $U_p$ as described in Lemma \ref{lemmaSimple}. Then,  by Proposition \ref{prop:Lip1}, $\mathrm{id}: (U_p, d_{g_R}) \to (U_p,\dhat_\tau)$ is a bi-Lipschitz map. Moreover, given any uniform Temple chart $\Phi_q: W_q \to \Phi_q (W_q)$ centered at $q\in U_p$, by Proposition \ref{prop:Lip2} we know that the restriction $\Phi_q:(\Phi_q^{-1}(U_p),d_{\mathbb{E}^{n+1}}) \to (U_p,d_{g_R})$ is bi-Lipschitz. Combining these two facts, we see that the restriction  $\Phi_q:(\Phi_q^{-1}(U_p),d_{\mathbb{E}^{n+1}}) \to (U_p,\dhat_\tau) $ is bi-Lipschitz, which proves the claim.
\end{proof}

We conclude with the following important corollary of Theorem \ref{thm:Lip}.

\begin{cor} 
    \label{zero-measure}
    Let $(N^{n+1},g)$ be a spacetime equipped with a weak temporal function \(\tau\). 
    Then the following holds:
    \begin{itemize}
    \item[(1)] $(N,\dhat_\tau)$ is a rectifiable metric space, covered by countably many bi-Lipschitz uniform Temple charts as described in Theorem \ref{thm:Lip}. 
   \item[(2)] The set $S$ of points where $\tau$ is not differentiable has $\mathcal{H}^{n+1}_{\hat{d}_\tau}(S)=0$,
   where $\mathcal{H}^{n+1}_{\hat{d}_\tau}$ is the $(n+1)$-dimensional Hausdorff measure of the metric space $(N,\hat{d}_\tau)$. 
    \end{itemize}
\end{cor}

\begin{proof}
    Since $N^{n+1}$ is a smooth manifold, a cover $\{U_p\}_{p\in N}$ by uniform Temple neighborhoods as defined in Lemma \ref{lemmaSimple} has a countable subcover $\{U_{p_i}\}_{i\in \mathbb{N}}$. Since all uniform Temple charts $\Phi_{p_i}: (\Phi_{p_i}^{-1}(U_{p_i}), d_{\mathbb{E}^{n+1}})\to (U_{p_i},\dhat_\tau)$ are bi-Lipschitz as shown in Theorem \ref{thm:Lip}, it follows that $(N,\dhat_\tau)$ is a countably rectifiable metric space, which proves the first statement. 
    
    As for the second statement, since $S$ is covered by countably many uniform Temple charts as described in the previous paragraph, it suffices to show that  for any uniform Temple chart $\Phi_{p}:\Phi_{p}^{-1}(U_{p})\to U_{p}$ as in Theorem \ref{thm:Lip} we have $\mathcal{H}^{n+1}_{\hat{d}_\tau} (S\cap U_p)= 0$. For this, we first note that the composition map $\tau \circ \Phi_{p} : \Phi_{p}^{-1}(U_{p})  \subset{\mathbb{R}^{n+1}} \to \mathbb{R}$ is Lipschitz so its singular set $\tilde{S}_p$ has 
    \be
        \mathcal{H}^{n+1}_{d_{\mathbb{E}^{n+1}}}(\tilde S_p)=\mathcal{L}^{n+1}(\tilde S_p) = 0
    \ee
    where $\mathcal{H}^{n+1}_{d_{\mathbb{E}^{n+1}}}$ is the Hausdorff measure of $\mathbb{R}^{n+1}$ with respect to the Euclidean distance and  $\mathcal{L}^{n+1}$ denotes the Lebesgue measure of $\mathbb{R}^{n+1}$. Next, we recall that 
    \be
        \Phi_p:(\Phi_p^{-1}(U_p),d_{\mathbb{E}^{n+1}}) \to (U_p,\dhat_\tau) 
    \ee 
    is bi-Lipschitz (see Theorem \ref{thm:Lip}), so we may conclude that $\mathcal{H}^{n+1}_{\hat{d}_\tau}(\Phi_p(\tilde S_p))=0$. Finally, we note that $S\cap U_p \subset \Phi_p(\tilde S_p) \cup \eta_p$, so we get 
    \be
        \mathcal{H}^{n+1}_{\hat{d}_\tau} (S\cap U_p) \leq \mathcal{H}^{n+1}_{\hat{d}_\tau}(\Phi_p(\tilde S_p)) + \mathcal{H}^{n+1}_{\hat{d}_\tau}(\eta_p) = 0
    \ee
    proving the claim.
\end{proof}

\subsection{Null Distance Encodes Causality on Uniform Temple Charts} \label{sec:null}

In this section we briefly  discuss yet another consequence of Theorem \ref{th-uniform-Temple}, namely the following result: 

\begin{thm} 
    \label{th-encodes-causality-locally}
    Let $(N^{n+1},g)$ be a Lorentzian manifold of dimension $n+1$, $n\geq 1$. Suppose that $\tau:N \to \mathbb{R}$ 
    is a weak temporal function in the sense of Definition \ref{def:WeakTemporalFunction}. 
    Then $\hat{d}_\tau$ locally encodes causality in the sense that about every point $p\in N$ there is a uniform Temple neighborhood $\widetilde U_p$ such that for all $q,q' \in \widetilde U_p$ we have
    \be\label{eq:encoding-}
        \hat{d}_\tau(q,q') = \tau(q')-\tau(q)\,  \iff \,  q' \textrm{ is in the causal future of } q.
    \ee
\end{thm}

\begin{proof} 
    We first prove that for every $p\in N$ there is a neighborhood $\widetilde U_p$ such that for all $q,q' \in \widetilde U_p$ we have
    \be
        \hat{d}_\tau(q,q') = \tau(q')-\tau(q) \, \Longrightarrow \, q' \textrm{ is in the causal future of } q.
    \ee
    
    Given $p\in N$, we begin by choosing its uniform Temple neighborhood $U_p$ as in Lemma \ref{lemmaSimple}. In this case for any $q\in U_p$ there is a uniform Temple chart $\Phi_q: W_r \to V_q$ such that $U_p \subset V_q=\Phi_q(W_r) \subset U$,  where $U$ is the neighborhood of $p$  as in Definition \ref{def:WeakTemporalFunction}. See Lemma \ref{lemmaSimple} for more details.
 
    Next, we choose a radius $r_p>0$  such that $B_{\dhat_\tau}  (p,4 r_p)\subset U_p$, and set $\widetilde U_p := B_{\dhat_\tau }(p,  r_p/2)$ so that for any $q\in \widetilde U_p$ we have 
    \be
        \widetilde U_p 
        = B_{\dhat_\tau} \left(p, \tfrac{r_p}{2}\right) 
        \subset B_{\dhat_\tau }(q, r_{p}) 
        \subset B_{\dhat_\tau }(q, 2r_{p})
        \subset B_{\dhat_\tau }(p, 4r_{p})
        \subset U_p
        \subset V_q 
        \subset U.
    \ee
    
    As a consequence, given any  $q,q'\in  \widetilde U_p$ with $\tau(q') - \tau(q) = \hat{d}_\tau(q,q')$ and any $\epsilon \in (0, r_p) $ we have $q'\in B_{\dhat_\tau}(q,r_p)$ and $B_{\dhat_\tau} (q, 2 r_p) \subset V_q$ so by \cite[Lemma 3.3]{Sak-Sor-null} there exists a piecewise causal curve in $B_{\dhat_\tau }  (q,  2r_{p})$ zigzagging from $q$ to  $q'$, with an $\epsilon$-controlled past directed part. See \cite[Lemma 3.3]{Sak-Sor-null} for a precise formulation and more details. Although this does not immediately imply that $q'\in J^+(q)$ (see the discussion in the beginning of Section III in \cite{Sak-Sor-null}), we may argue as in \cite[Section III.C, Proof of Theorem 1.1]{Sak-Sor-null}, to show that $\omega_q (q') \geq 0$, where $\omega_q$ is the optical function of the chart $\Phi_q:W_r\to V_q$. This implies that $q' \in J^+(q)$ whenever $q,q'\in \widetilde U_p$ satisfy $\tau(q') - \tau(q) = \hat{d}_\tau(q,q')$.  
    
    To conclude the proof, we only have to recall that  if $q,q'\in \widetilde U_p$ are such that $q'\in J^+(q)$ then $\hat{d}_\tau(q,q') = \tau(q')-\tau(q)$ holds, see \cite[Corollary 3.19]{SV-null}.
\end{proof}

\begin{rmrk}
    This result may be seen as an improvement of \cite[Theorem 1.1]{Sak-Sor-null} where \eqref{eq:encoding-} was shown to hold with $q=p$. Note also that \cite[Remark 3.8]{Burtscher-Garcia-Heveling-22} discusses an alternative approach to obtaining the result of Theorem \ref{th-encodes-causality-locally}.  At the same time we would like to point out that \cite[Theorem 1.1]{Sak-Sor-null} is a stronger result than Theorem \ref{th-encodes-causality-locally} in the sense that it holds for \emph{generalized} time functions satisfying the locally anti-Lipschitz condition, whereas in Theorem \ref{th-encodes-causality-locally} we require that the time function is weak temporal. See Definition \ref{def:WeakTemporalFunction} which clarifies the difference between the notions of locally anti-Lipschitz (generalized) time functions and weak temporal functions. 
\end{rmrk}

\subsection{A Lorentzian Isometry Theorem} 
\label{sec:Isometry}
We conclude the paper by showing that a spacetime equipped with a weak temporal function $\tau$ such that $|\nabla \tau|_g=1$ almost everywhere can be converted into a metric space in a canonical way. More specifically, we prove the following result.
\begin{thm} 
    \label{Lorentzian-isom}
    Let $(N_1, g_1, \tau_1)$ and $(N_2, g_2, \tau_2)$ be two $(n+1)$-dimensional Lorentzian manifolds, where $n \geq 2$, equipped with weak temporal functions $\tau_i$ such that
    \be \label{grad-tau-one}
        |\nabla^{g_i} \tau_i|_{g_i}=1  \text{ almost everywhere}, \quad i=1,2. 
    \ee
    If there exists a bijection $F: N_1 \to N_2$ that preserves null distances,
    \be 
        \label{hat-here}
        \hat{d}_{\tau_1} (p,q) = \hat{d}_{\tau_2} (F(p), F(q)) \quad  \text{for any} \quad p, q \in N_1,
    \ee
    and time functions,
    \be \label{tau-here}
        \tau_1 = \tau_2 \circ F,
    \ee 
    then $F$ is a diffeomorphism and a Lorentzian isometry, $F^* g_2 =  g_1$.
\end{thm}
    

Previously Theorem \ref{Lorentzian-isom} was proven by Sakovich and Sormani in \cite[Theorem 1.3]{Sak-Sor-null} under the additional assumption that the \emph{causality of $(N_i,g_i)$, $i=1,2$, is globally encoded by $\tau_i$ and  $\dhat_{\tau_i}$} in the sense that we have
\be 
    \label{eq-global}
    \hat{d}_{\tau_i}(p,q)=\tau_i(q)-\tau_i(p) \iff q \in J^+ (p) \quad \text{holds for all}\quad  p,q\in N_i.
\ee
We note that this additional assumption is satisfied, for example, in the case when the level sets of the time function are future causally complete, as shown by Galloway in \cite[Theorem 3]{Galloway-Null} (see also \cite[Theorem 4.1]{Sak-Sor-null} and \cite[Corollary 1.10]{Burtscher-Garcia-Heveling-22} for earlier results). At the same time, dropping the completness assumption it is very easy to construct examples in which global encoding of causality fails. See for instance \cite[Example 2.2]{Sak-Sor-null} where Minkowski spacetime with its natural time function and a half line removed is considered. In contrast to  \cite[Theorem 1.3]{Sak-Sor-null}, Theorem \ref{Lorentzian-isom} does not require global encoding of causality, as it turns out that local encoding of causality ensured by  Theorem \ref{th-encodes-causality-locally} suffices for proving the result.

The proof of \thmref{Lorentzian-isom} is similar to that of \cite[Theorem 4.1]{Sak-Sor-null}. In particular, it uses the following theorem of Levichev \cite{Levichev-1987}, which builds upon a celebrated result of Hawking (cf. \cite[Lemma 19]{Hawking-2014}). An excellent overview of these and related results, in particular those of Zeeman \cite{Zeeman-64}, Hawking-King-McCarthy \cite{Hawking-King-McCarthy-1976}, and Malament \cite{Malament-1977}), is provided in Minguzzi \cite[Section 4.3.4]{Minguzzi-2019}.

\begin{thm} 
    \label{thLevichev}
    Let $(N_1,g_1)$ and $(N_2,g_2)$ be two $(n+1)$-dimensional distinguishing spacetimes, where $n\geq 2$, and let $F: N_1 \to N_2$ be a causal bijection, that is a bijection such that
    \be
        q \in J^+(p) \, \Longleftrightarrow \, F(q) \in F(J^+(p)).
    \ee
    Then $F$ is a smooth conformal isometry, i.e. there exists a smooth function $\phi: N_1 \to (0,\infty)$ such that $F^* g_2 = \phi^2 g_1$.
\end{thm} 

Another important ingredient in the proof of \thmref{Lorentzian-isom} is the following lemma based on the local encoding of causality established in \thmref{th-encodes-causality-locally}.

\begin{lem} 
    \label{lem-causal-bi}
    Under the assumptions of Theorem~\ref{Lorentzian-isom}, $F$ is {\bf locally a causal bijection}, that is,  for any $p\in N_1$ there exists a neighborhood $U$, such that
    \be 
        \label{causal-bijection-1}
        F(J^+(q,U)) = J^+(F(q),F(U)) \quad \text{for all} \quad q \in U.
    \ee
\end{lem}

\begin{proof}
    
    Let \( p \in N_1\). By Theorem~\ref{th-encodes-causality-locally} there exists an open neighborhood \( W_1 \subset N_1 \) about \( p \) and an open neighborhood \( W_2 \subset N_2 \) about \( F(p) \) such that for all \( q_i,q_i' \in W_i \) we have
    \be 
        \label{q'q}
        q_i' \in J^+(q_i) \iff \hat{d}_{\tau_i}(q_i', q_i) = \tau_i(q_i')-\tau_i(q_i).
    \ee
    Since \( F\) is distance preserving (see \eqref{hat-here}) it is a homeomorphism, hence \( F^{-1}(W_2) \subset N_1 \) is an open set containing \( p \). Due to (4) in Theorem~\ref{th-uniform-Temple} we may without loss of generality assume that \( W_1 \subset F^{-1}(W_2) \). We will show that \( U = W_1 \) satisfies \eqref{causal-bijection-1}.

    We will first prove that 
    \be 
        \label{causal-bijection-5}
        F(J^+(q,W_1)) \subset J^+(F(q),F(W_1)) \quad \text{for all} \quad q \in W_1.
    \ee
    For this, let \( q\in W_1 \) and \( z \in J^+(q,W_1) \) be arbitrary. Then there exists a future causal 
    curve \( \gamma: [0,1] \to W_1 \) such that \( \gamma(0) = q \) and \( \gamma(1) = z \). Recalling \eqref{q'q}, \eqref{hat-here} and \eqref{tau-here} we conclude that for all \( t,t' \in [0,1]\) we have
    \be
        \label{eq:Implication}
        \begin{split}
        t \leq t' 
        & \implies \gamma(t') \in J^+(\gamma(t),W_1) \\
        & \implies \gamma(t') \in J^+(\gamma(t)) \\
        & \implies \hat d_{\tau_1}(\gamma(t'),\gamma(t)) = \tau_1(\gamma(t')) - \tau_1(\gamma(t)) \\
        & \implies \hat d_{\tau_2}((F \circ \gamma)(t'),(F \circ \gamma)(t)) = \tau_2((F\circ\gamma)(t')) - \tau_2((F \circ \gamma)(t)) \\
        & \implies (F \circ \gamma)(t') \in J^+((F\circ \gamma)(t)).
        \end{split}
    \ee
    We will now deform the continuous function \( F\circ \gamma: [0,1] \to F(W_1) \) into a (piecewise smooth) future causal curve from $F(q)$ to $F(z)$ that is completely contained in \( F(W_1) \). 
    
    Since the spacetimes $(N_i,g_i)$ are equipped with time functions $\tau_i: N_i \to \mathbb{R}$, they are stably causal, and in particular strongly causal, see e.g. \cite[Theorem 3.56 and Proposition 3.57]{MinguzziSanchez}. Consequently, since \( F(W_1) \) is open, for every \( t \in [0,1] \) there is a causally convex neighborhood \( \widetilde{W}_{(F\circ\gamma )(t)} \subset F(W_1) \) of \( (F \circ \gamma)(t) \), so that every causal curve $\beta: [0,1] \to N_2$ such that $\beta(0),\beta(1) \in \widetilde{W}_{(F\circ\gamma )(t)}$ has image contained in $\widetilde{W}_{(F\circ\gamma )(t)}$.
    
    For each \( t \in [0,1] \), we let \( I(t) \) denote the connected component of the preimage \( (F \circ \gamma)^{-1}(\widetilde W_{(F \circ \gamma)(t)}) \) that contains the point \( t \). By definition, \( \{I(t)\}_{t \in [0,1]} \) is a collection of open sets in \( [0,1] \) such that  \( \cup_{t\in[0,1]}I(t)=[0,1] \). Since \( [0,1] \) is compact there is a finite collection of intervals \( \{I(t_i)\}_{i = 1}^n \) such that \( \cup_{i=1}^n I(t_i)=[0,1] \).
     
    Due to Lemma \ref{lem:CoveringLemma} there is $L\in \{1,\ldots, n\}$ and points 
    \be
    0=s_0 < s_1 < \ldots < s_{L} = 1 
    \ee
    such that for all $l\in\{1,\ldots, L\}$ there is $i_l\in \{1,\ldots, n\}$ such that $s_{l-1}, s_l \in I(t_{i_l})$. Since \( s_{l-1} < s_l \),  \eqref{eq:Implication} implies that
    \be 
        (F\circ \gamma)(s_l) \in J^+((F\circ \gamma)(s_{l-1}))
        \quad \text{for all \( l = 1,\ldots, L \)}.
    \ee
    At the same time, since \( s_{l-1}, s_l \in I(t_{i_l}) \), we have
    \be
        (F \circ \gamma)(s_{l-1}), \, (F \circ \gamma)(s_l) \in (F \circ \gamma)(I(t_{i_l})) \subset \widetilde W_{(F \circ \gamma)(t_{i_l})} \quad \text{for all \( l = 1,\ldots, L \)}.
    \ee
    By properties of causally convex neighborhoods \( \widetilde{W}_{(F\circ\gamma)(t)} \), we conclude that for all \( l = 1,\ldots, L \), there is a smooth future causal curve from \( (F\circ \gamma)(s_{l-1}) \) to \( (F\circ \gamma)(s_l) \) contained in \( \widetilde{W}_{(F\circ\gamma)(t_{i_l})} \subset F(W_1) \). Concatenating these curves one obtains a piecewise smooth, future causal curve from \( F(q) = (F \circ \gamma)(0) \) to \( F(z) = (F\circ \gamma)(1) \), wholly contained in \( F(W_1) \). Thus \( F(z) \in J^+(F(q),F(W_1)) \). Since \( z \in J^+(q,W_1) \) was arbitrary, \eqref{causal-bijection-5} follows. 

    The proof of the reverse inclusion
    \be 
        F(J^+(q,W_1)) \supset J^+(F(q),F(W_1)) \quad \text{for all} \quad q \in W_1,
    \ee
    is very similar to the proof of \eqref{causal-bijection-5} and is therefore omitted.
\end{proof}

\begin{proof}[Proof of Theorem~\ref{Lorentzian-isom}]
    \lemref{lem-causal-bi} implies that for every $p\in N_1$ there is a neighborhood $U_p$ such that $F: U_p \to F(U_p)$  is a causal bijection. By the continuity of the time functions $\tau_i$, $i=1,2$, both $U_p$ and $F(U_p)$ are distinguishing, see for example \cite[Theorem 4.5
    8 (v')]{Minguzzi-2019}. Applying \thmref{thLevichev} we conclude that $F: U_p \to F(U_p)$ is a conformal isometry so that $F^* g_2 = \phi^2 \, g_1$ on $U_p$ for a smooth function $\phi: U_p \to (0,\infty)$.

    That $\phi \equiv 1$, so that $F: U_p \to F(U_p)$ is an isometry, follows from the same argument using the coarea formula as in \cite[Proof of Theorem 1.3]{Sak-Sor-null}\footnote{We would like to point out a typo in the proof of \cite[Theorem 1.3]{Sak-Sor-null}: while the theorem is formulated for $\dim N_1= \dim N_2= n+1$ where $n\geq 2$ the proof of is written for $\dim N_1 = \dim N_2 =n \geq 3$.} We see that $F$ is a local isometry, but since it is also a bijection it is in fact a global isometry, which completes the proof.
\end{proof} 

Concluding this section we would like to point out that an assumption along the lines of \eqref{grad-tau-one} in \thmref{Lorentzian-isom} is necessary, see for example \cite[Example 5.2]{Sak-Sor-null}. We would also like to emphasize once again that \thmref{Lorentzian-isom} applies to a much larger class of spacetimes and time functions than \cite[Theorem 1.3]{Sak-Sor-null}, in particular to those  where there is no global encoding of causality, cf.  \cite[Example 2.2]{Sak-Sor-null}. All that is required is that the time function is weak temporal with $|\nabla \tau^g|_g=1$ almost everywhere.

\appendix

\section{}

In this appendix, we prove an elementary lemma about finite open covers of compact intervals.
\begin{lem}
    \label{lem:CoveringLemma}
    Let \( \{I_i\}_{i = 1}^n \), \( n \geq 1 \), be a collection of intervals  $I_i$ that are open in \( [0,1] \) and such that \( \cup_{i=1}^n I_i = [0,1] \). Then there exists $L\in \{1,\ldots, n\}$ and points 
    \be\label{eqIncreasing}
        0 = s_0 < s_1 < \ldots < s_{L} = 1 
    \ee
    such that for all $l\in\{1,\ldots, L\}$ there is $i_l \in \{1,\ldots, n\}$ such that $s_{l-1}, s_l \in I_{i_l}$.
\end{lem}

\begin{proof}
    The construction of the sequence \( \{s_l\}_{l = 0}^L \) proceeds in two steps. 
    
    \textbf{\emph{Step 1.}} We begin by relabeling the given collection of intervals in a particular order, thereby constructing a new (possibly smaller) collection \( \{\tilde I_l\}_{l = 1}^L \) covering $[0,1]$.  First, we let \( \tilde I_1 \) be any of the intervals in \( \{I_i\}_{i = 1}^n \) that contains \( 0 \). If \( 1 \in \tilde I_1 \) then we set $L=1$ and in this case \( \{\tilde I_l\}_{l = 1}^1 \) is the desired collection. Otherwise, we let \( \tilde I_2 \) be any of the intervals in \( \{I_i\}_{i = 1}^n \setminus \{\tilde I_1\}  \) that contains \( \sup (\tilde I_1) \). Once again, if \( 1 \in \tilde I_2 \) then we set $L=2$ and in this case \( \{\tilde I_l\}_{l = 1}^2 \) is the desired collection. Otherwise we let \( \tilde I_3 \) be any of the intervals in \( \{I_i\}_{i = 1}^n \setminus \{\tilde I_l\}_{l = 1}^2 \)  that contains \( \sup (\tilde I_2) \) and so on. Since the collection \( \{I_i\}_{i = 1}^n \) is finite, this process will terminate after finitely many steps, resulting in a subcollection \( \{\tilde I_l\}_{l = 1}^L \subset \{I_i\}_{i = 1}^n\), and it is an easy matter to check that $\cup_{l=1}^L \tilde I_l = [0,1]$. 
    

    \textbf{\emph{Step 2}.} Now that the sequence \( \{\tilde I_l\}_{l = 1}^L \) has been constructed, we can find a sequence \( \{s_l\}_{l = 0}^L \) with properties as in the statement of the lemma. If \( L = 1 \), then \( \tilde I_1 = [0,1] \) and we can let \( s_0 = 0 \) and \( s_1 = 1 \). If \( L > 1 \), then by construction, \( 0 \in \tilde I_1 \), \( 1 \in \tilde I_L \), and for all \( l = 1,\ldots, L - 1 \) we have  \( 0 < \sup(\tilde I_l) < \sup(\tilde I_{l + 1}) \leq 1 \). Moreover, it is straightforward to check that there is an $\epsilon > 0$ such that $\sup(\tilde I_l)-\epsilon \in \tilde I_l \cap \tilde I_{l + 1}$ for all \( l = 1,\ldots, L - 1\).  With this at hand, we see that the sequence $\{s_l\}_{l=0}^L$ defined by \( s_0 = 0 \), \( s_L = 1 \) and $s_l := \sup(\tilde I_{l})-\varepsilon$ for $l=1,\ldots, L - 1$ is strictly increasing so that \eqref{eqIncreasing} holds. Furthermore, we have $s_l, s_{l-1}\in \tilde I_{l}$ for $l=1,\ldots, L$, proving the claim.
    \end{proof}
    

\bibliographystyle{plain}
\bibliography{SF-bib.bib}

\begin{thebibliography}{10}

\bibitem{Allen-23}
Brian Allen.
\newblock Null distance and {G}romov-{H}ausdorff convergence of warped product
  spacetimes.
\newblock {\em Gen. Relativity Gravitation}, 55(10):Paper No. 118, 34, 2023.

\bibitem{Allen-Burtscher-20}
Brian Allen and Annegret Burtscher.
\newblock Properties of the null distance and spacetime convergence.
\newblock {\em Int. Math. Res. Not. IMRN}, 10:7729--7808, 2022.

\bibitem{AGH}
Lars Andersson, Gregory~J. Galloway, and Ralph Howard.
\newblock The cosmological time function.
\newblock {\em Classical Quantum Gravity}, 15(2):309--322, 1998.

\bibitem{Burtscher-Garcia-Heveling-22}
Annegret Burtscher and Leonardo Garc\'{\i}a-Heveling.
\newblock Global hyperbolicity through the eyes of the null distance.
\newblock {\em Comm. Math. Phys.}, 405(4):Paper No. 90, 35, 2024.

\bibitem{Burtscher-12}
Annegret~Y. Burtscher.
\newblock Length structures on manifolds with continuous {R}iemannian metrics.
\newblock {\em New York J. Math.}, 21:273--296, 2015.

\bibitem{Che-Per-Sor}
Mauricio Che, Raquel Perales, and Christina Sormani.
\newblock {G}romov's {C}ompactness {T}heorem for the {I}ntrinsic
  {T}imed-{H}ausdorff {D}istance, 2026.
\newblock https://arxiv.org/abs/2510.13069.

\bibitem{CGM}
Piotr~T. Chru\'{s}ciel, James D.~E. Grant, and Ettore Minguzzi.
\newblock On differentiability of volume time functions.
\newblock {\em Ann. Henri Poincar\'{e}}, 17(10):2801--2824, 2016.

\bibitem{DoCarmo}
Manfredo Perdig\~ao do~Carmo.
\newblock {\em Riemannian geometry}.
\newblock Mathematics: Theory \& Applications. Birkh\"auser Boston, Inc.,
  Boston, MA, portuguese edition, 1992.

\bibitem{Galloway-Null}
Gregory~J. Galloway.
\newblock A note on null distance and causality encoding.
\newblock {\em Classical Quantum Gravity}, 41(1):Paper No. 017001, 5, 2024.

\bibitem{Graf-Sormani}
Melanie Graf and Christina Sormani.
\newblock Lorentzian area and volume estimates for integral mean curvature
  bounds.
\newblock In {\em Developments in {L}orentzian geometry}, volume 389 of {\em
  Springer Proc. Math. Stat.}, pages 105--128. Springer, Cham, [2022]
  \copyright 2022.

\bibitem{Hawking-2014}
Stephen~W. Hawking.
\newblock Singularities and the geometry of spacetime: reprint of the {A}dams
  {P}rize essay 1966.
\newblock {\em Eur Phys J H}, 39:413–503, 2014.

\bibitem{Hawking-King-McCarthy-1976}
Stephen~W. Hawking, Andrew~R. King, and Patrick~J. McCarthy.
\newblock A new topology for curved space -- time which incorporates the
  causal, differential, and conformal structures.
\newblock {\em Journal of Mathematical Physics}, 17(2):174--181, 1976.

\bibitem{KobayashiNomizu}
Shoshichi Kobayashi and Katsumi Nomizu.
\newblock {\em Foundations of differential geometry. {V}ol. {I}}.
\newblock Wiley Classics Library. John Wiley \& Sons, Inc., New York, 1996.
\newblock Reprint of the 1963 original, A Wiley-Interscience Publication.

\bibitem{Kunzinger-Steinbauer-21}
Michael Kunzinger and Roland Steinbauer.
\newblock Null distance and convergence of {L}orentzian length spaces.
\newblock {\em Ann. Henri Poincar\'{e}}, page 1—24, 2022.

\bibitem{Lee}
John~M. Lee.
\newblock {\em Introduction to {R}iemannian manifolds}, volume 176 of {\em
  Graduate Texts in Mathematics}.
\newblock Springer, Cham, second edition, 2018.

\bibitem{Levichev-1987}
Aexander~V. Levichev.
\newblock The causal structure of a {L}orentzian manifold determines its
  conformal geometry.
\newblock {\em Dokl. Akad. Nauk SSSR}, 293(6):1301--1305, 1987.

\bibitem{Malament-1977}
David~B. Malament.
\newblock The class of continuous timelike curves determines the topology of
  spacetime.
\newblock {\em J. Mathematical Phys.}, 18(7):1399--1404, 1977.

\bibitem{Minguzzi-2019}
Ettore Minguzzi.
\newblock Lorentzian causality theory.
\newblock {\em Living reviews in relativity}, 22(1):1--202, 2019.

\bibitem{MinguzziSanchez}
Ettore Minguzzi and Miguel S\'anchez.
\newblock The causal hierarchy of spacetimes.
\newblock In {\em Recent developments in pseudo-{R}iemannian geometry}, ESI
  Lect. Math. Phys., pages 299--358. Eur. Math. Soc., Z\"urich, 2008.

\bibitem{MisnerThorneWheeler}
Charles~W. Misner, Kip~S. Thorne, and John~Archibald Wheeler.
\newblock {\em Gravitation}.
\newblock W. H. Freeman and Co., San Francisco, CA, 1973.

\bibitem{O'neill-text}
Barrett O'Neill.
\newblock {\em Semi-{R}iemannian geometry}, volume 103 of {\em Pure and Applied
  Mathematics}.
\newblock Academic Press, Inc. [Harcourt Brace Jovanovich, Publishers], New
  York, 1983.
\newblock With applications to relativity.

\bibitem{future-work}
Anna Sakovich and Christina Sormani.
\newblock Space-time intrinsic flat convergence.
\newblock {\em In progress}.

\bibitem{Sak-Sor-null}
Anna Sakovich and Christina Sormani.
\newblock The null distance encodes causality.
\newblock {\em J. Math. Phys.}, 64(1):Paper No. 012502, 18, 2023.

\bibitem{Sak-Sor-24}
Anna Sakovich and Christina Sormani.
\newblock Introducing various notions of distances between space-times.
\newblock {\em arXiv:2410.16800}, 2024.

\bibitem{Sormani-Oberwolfach-18}
Christina Sormani.
\newblock Oberwolfach report 2018: {S}pacetime intrinsic flat convergence.
\newblock {\em Oberwolfach Reports}, 2018.

\bibitem{SV-null}
Christina Sormani and Carlos Vega.
\newblock Null distance on a spacetime.
\newblock {\em Classical Quantum Gravity}, 33(8):085001, 29, 2016.

\bibitem{Temple-1938}
George Temple.
\newblock New systems of normal co-ordinates for relativistic optics.
\newblock {\em Proceedings of the Royal Society of London. Series A.
  Mathematical and Physical Sciences}, 168(932):122--148, 1938.

\bibitem{Vega21}
Carlos Vega.
\newblock Spacetime distances: an exploration.
\newblock {\em arXiv:2103.01191}, 2021.

\bibitem{Wald-Yip}
Robert~M. Wald and Ping Yip.
\newblock On the existence of simultaneous synchronous coordinates in
  spacetimes with spacelike singularities.
\newblock {\em J. Math. Phys.}, 22:2659–2665, 1981.

\bibitem{Zeeman-64}
Erik~C. Zeeman.
\newblock Causality implies the {L}orentz group.
\newblock {\em J. Math. Phys. 5}, 490, 1964.

\end{thebibliography}

\end{document}